\documentclass[twoside]{article} 
\usepackage[accepted]{aistats2017}

\usepackage[round]{natbib} 
\usepackage[utf8]{inputenc} 
\usepackage[T1]{fontenc}
\usepackage{hyperref}       
\usepackage{url}            
\usepackage{booktabs}       
\usepackage{amsfonts}       
\usepackage{nicefrac}       
\usepackage{microtype}      

\usepackage{enumitem} 

\usepackage{color}
\hypersetup{ %
    pdftitle={},
    pdfauthor={},
    pdfkeywords={},
    pdfborder=0 0 0,
	pdfpagemode=UseNone,
    colorlinks=true,
    linkcolor=blue, 
    citecolor=blue, 
    filecolor=blue, 
    urlcolor=blue, 
    pdfview=FitH,
    pdfauthor={Rémi Leblond, Fabian Pedregosa, Simon Lacoste-Julien},
} 

\usepackage{graphicx}
\usepackage{subcaption}
\usepackage{float}

\usepackage{algorithm}
\usepackage{algorithmic}

\usepackage{mdframed}

\usepackage{amssymb, amsmath, amsthm}
\usepackage{bbm, bm}

\newcommand{\stepsize}{\gamma}
\newcommand{\strongconvex}{\mu}
\newcommand{\overlap}{\tau}
\newcommand{\contraction}{\rho}
\newcommand{\sparsity}{\Delta}
\newcommand{\sparsityr}{\Delta_r}
\newcommand{\lipschitz}{L}
\newcommand{\lyapunov}{\mathcal{L}}

\newcommand{\E}{\mathbb{E}}
\newcommand{\Econd}{\mathbf{E}}
\newcommand{\ind}{\mathbbm{1}}

\newcommand{\ASAGA}{\textsc{Asaga}}
\newcommand{\SAGA}{\textsc{Saga}}
\newcommand{\SAG}{\textsc{Sag}}

\newcommand{\SVRG}{\textsc{Svrg}}
\newcommand{\Hogwild}{\textsc{Hogwild}}
\newcommand{\SDCA}{\textsc{Sdca}}
\newcommand{\SGD}{\textsc{Sgd}}
\newcommand{\HSAG}{\textsc{Hsag}}
\newcommand{\KROMAGNON}{\textsc{Kromagnon}}

\newtheorem{asm}{Assumption}
\newtheorem{prop}{Property}
\newtheorem{proposition}{Proposition}
\newtheorem{lemma}{Lemma}
\newtheorem{theorem}{Theorem}
\newtheorem{corollary}[theorem]{Corollary}
\newtheorem{definition}{Definition}
\newtheorem{remark}{Remark}


\begin{document}
\twocolumn[
	
\aistatstitle{\ASAGA: Asynchronous Parallel \SAGA}


	
\aistatsauthor{Rémi Leblond
\And
Fabian Pedregosa
\And
Simon Lacoste-Julien}

\aistatsaddress{ INRIA - Sierra Project-team\\
\'Ecole normale sup\'erieure, Paris 
\And INRIA - Sierra Project-team\\
\'Ecole normale sup\'erieure, Paris
\And
Department of CS \& OR (DIRO) \\
Universit\'e de Montr\'eal, Montr\'eal}

]

\newcommand{\fix}{\marginpar{FIX}}
\newcommand{\new}{\marginpar{NEW}}

\begin{abstract}
\vspace{-2mm}
We describe \ASAGA, an asynchronous parallel version of the incremental gradient algorithm \SAGA\ that enjoys fast linear convergence rates. 
Through a novel perspective, we revisit and clarify a subtle but important technical issue present in a large fraction of the recent convergence rate proofs for asynchronous parallel optimization algorithms, and propose a simplification of the recently introduced ``perturbed iterate'' framework that resolves it. 
We thereby prove that \ASAGA\ can obtain a theoretical linear speedup on multi-core systems even without sparsity assumptions. 
We present results of an implementation on a 40-core architecture illustrating the practical speedup as well as the hardware overhead. 
\end{abstract} 

\vspace{-2mm}
\section{Introduction}
We consider the unconstrained optimization problem of minimizing a \emph{finite sum} of smooth convex functions:
\vspace{-1mm}
\begin{equation} \label{eq:finiteSum}
\min_{x \in \mathbb{R}^d} f(x), \quad f(x) := \frac{1}{n} \sum_{i=1}^{n} f_i(x),
\end{equation}
where each $f_i$ is assumed to be convex with $\lipschitz$-Lipschitz continuous gradient, $f$ is $\mu$-strongly convex and $n$ is large (for example, the number of data points in a regularized empirical risk minimization setting).
We define a condition number for this problem as $\kappa := \nicefrac{\lipschitz}{\mu}$.
A flurry of randomized incremental algorithms (which at each iteration select $i$ at random and process only one gradient $f'_i$) have recently been proposed to solve~\eqref{eq:finiteSum} with a fast\footnote{Their complexity in terms of gradient evaluations to reach an accuracy of $\epsilon$ is $O((n+\kappa)\log(\nicefrac{1}{\epsilon}))$, in contrast to $O(n\kappa\log(\nicefrac{1}{\epsilon}))$ for batch gradient descent in the worst case.} linear convergence rate, such as \SAG~\citep{SAG}, \SDCA~\citep{SDCA}, \SVRG~\citep{svrg} and \SAGA~\citep{SAGA}.
These algorithms can be interpreted as variance reduced versions of the stochastic gradient descent (\SGD) algorithm, and they have demonstrated both theoretical and practical improvements over \SGD\ (for the \emph{finite sum} optimization problem~\eqref{eq:finiteSum}).

In order to take advantage of the multi-core architecture of modern computers, the aforementioned optimization algorithms need to be adapted to the asynchronous parallel setting, where multiple threads work concurrently.
Much work has been devoted recently in proposing and analyzing asynchronous parallel variants of algorithms such as \SGD~\citep{hogwild}, \SDCA~\citep{asyncSDCA2015} and \SVRG~\citep{smola,mania,asySVRG}.
Among the incremental gradient algorithms with fast linear convergence rates that can optimize~\eqref{eq:finiteSum} in its general form, only \SVRG\ has had an asynchronous parallel version proposed.\footnote{We note that \SDCA\ requires the knowledge of an explicit $\mu$-strongly convex regularizer in~\eqref{eq:finiteSum}, whereas \SAG~/ \SAGA\ are adaptive to any local strong convexity of $f$~\citep{laggedsaga,SAGA}. This is also true for a variant of \SVRG~\citep{qsaga}.}
No such adaptation has been attempted yet for \SAGA, even though one could argue that it is a more natural candidate as, contrarily to \SVRG, it is not epoch-based and thus has no synchronization barriers at all.

\vspace{-2mm}
\paragraph{Contributions.} In Section~\ref{scs:sparse_saga}, we present a novel sparse variant of \SAGA\ that is more adapted to the parallel setting than the original \SAGA\ algorithm. 
In Section~\ref{sec:ASAGA}, we present \ASAGA, a lock-free asynchronous parallel version of Sparse \SAGA\ that does not require consistent reads. 
We propose a simplification of the ``perturbed iterate'' framework from~\citet{mania} as a basis for our convergence analysis. 
At the same time, through a novel perspective, we revisit and clarify a technical problem present in a large fraction of the literature on randomized asynchronous parallel algorithms (with the exception of~\citet{mania}, which also highlights this issue): namely, they all assume unbiased gradient estimates, an assumption that is inconsistent with their proof technique without further synchronization assumptions. 
In Section~\ref{ssec:convergence}, we present a tailored convergence analysis for \ASAGA. Our main result states that \ASAGA\ obtains the same geometric convergence rate per update as \SAGA\ when the overlap bound $\overlap$ (which scales with the number of cores) satisfies
$\overlap \leq \mathcal{O}(n)$ and $\overlap \leq \mathcal{O}({\scriptstyle \frac{1}{\sqrt{\sparsity}}} \max\{1,\frac{n}{\kappa} \})$, where $\sparsity \leq 1$ is a measure of the sparsity of the problem, notably implying that a linear speedup is theoretically possible even without sparsity in the well-conditioned regime where $n \gg \kappa$.
In Section~\ref{sec:results}, we provide a practical implementation of \ASAGA\ and illustrate its performance on a 40-core architecture, showing improvements compared to asynchronous variants of \SVRG\ and \SGD.

\vspace{-2mm}
\paragraph{Related Work.}
The seminal textbook of~\citet{bertsekasParalle1989} provides most of the foundational work for parallel and distributed optimization algorithms. 
An asynchronous variant of \SGD\ with constant step size called \Hogwild\ was presented by~\citet{hogwild}; part of their framework of analysis was re-used and inspired most of the recent literature on asynchronous parallel optimization algorithms with convergence rates, including asynchronous variants of coordinate descent~\citep{asyncCD2015}, \SDCA~\citep{asyncSDCA2015}, \SGD\ for non-convex problems~\citep{taming,asyncSGDNonConvex2015}, \SGD\ for stochastic optimization~\citep{duchi} and \SVRG~\citep{smola,asySVRG}. 
These papers make use of an unbiased gradient assumption that is not consistent with the proof technique, and thus suffers from technical problems\footnote{Except~\citet{duchi} that can be easily fixed by incrementing their global counter \emph{before} sampling.} that we highlight in Section~\ref{ssec:labelingIssue}. 

The ``perturbed iterate'' framework presented in~\citet{mania} is to the best of our knowledge the only one that does not suffer from this problem, and our convergence analysis builds heavily from their approach, while simplifying it.
In particular, the authors assumed that $f$ was both strongly convex and had a bound on the gradient, two \emph{inconsistent} assumptions in the unconstrained setting that they analyzed.
We overcome these difficulties by using tighter inequalities that remove the requirement of a bound on the gradient. We also propose a more convenient way to label the iterates (see Section~\ref{ssec:labelingIssue}).
The sparse version of \SAGA\ that we propose is also inspired from the sparse version of \SVRG\ proposed by~\citet{mania}. 
\citet{smola} presents a hybrid algorithm called \HSAG\ that includes \SAGA\ and \SVRG\ as special cases. 
Their asynchronous analysis is epoch-based though, and thus does not handle a fully asynchronous version of \SAGA\ as we do. 
Moreover, they require consistent reads and do not propose an efficient sparse implementation for \SAGA, in contrast to \ASAGA.  

\vspace{-2mm}
\paragraph{Notation.}  We denote by $\E$ a full expectation with respect to all the randomness, and by $\Econd$ the \emph{conditional} expectation of a random~$i$ (the index of the factor $f_i$ chosen in \SGD-like algorithms), conditioned on all the past, where ``past'' will be clear from the context. 
$[x]_v$ is the coordinate~$v$ of the vector~$x \in \mathbb{R}^d$.
$x^+$ represents the updated parameter vector after one algorithm iteration.

\vspace{-1mm}
\section{Sparse \SAGA}\label{scs:sparse_saga}
\vspace{-3mm}
Borrowing our notation from~\citet{qsaga}, we first present the original \SAGA\ algorithm and then describe a novel sparse variant that is more appropriate for a parallel implementation. 

\vspace{-2mm}
\paragraph{Original \SAGA\ Algorithm.} The standard \SAGA\ algorithm~\citep{SAGA} maintains two moving quantities to optimize~\eqref{eq:finiteSum}: the current iterate~$x$ and a table (memory) of historical gradients $(\alpha_i)_{i=1}^n$.\footnote{For linear predictor models, the memory $\alpha_i^0$ can be stored as a scalar. Following~\cite{qsaga}, $\alpha_i^0$ can be initialized to any convenient value (typically $0$), unlike the prescribed $f'_i(x_0)$ analyzed in~\citet{SAGA}.}
At every iteration, the \SAGA\ algorithm samples uniformly at random an index $i \in \{1,\ldots, n\}$, and then executes the following update on~$x$ and~$\alpha$ (for the unconstrained optimization version):
\begin{equation}\label{eq:SAGAupdate}
x^{+} = x - \stepsize \big(f'_i(x) - \alpha_i + \bar \alpha\big); \qquad  \alpha_i^+ = f'_i(x),
\end{equation}
where $\stepsize$ is the step size and $\bar \alpha := \nicefrac{1}{n} \sum_{i=1}^n \alpha_i$ can be updated efficiently in an online fashion. Crucially, $\Econd \alpha_i = \bar \alpha$ and thus the update direction is unbiased ($\Econd x^{+} = x - \stepsize f'(x)$). 
Furthermore, it can be proven (see~\citet{SAGA}) that under a reasonable condition on $\stepsize$, the update has vanishing variance, which enables the algorithm to converge linearly with a constant step size.

\vspace{-2mm}
\paragraph{Motivation for a Variant.}
In its current form, every \SAGA\ update is dense even if the individual gradients are sparse due to the historical gradient ($\bar \alpha$) term. 
\citet{laggedsaga} introduced a special implementation with lagged updates where every iteration has a cost proportional to the size of the support of $f_i'(x)$. 
However, this subtle technique is not easily adaptable  to the parallel setting (see App.~\ref{apx:DifficultyLagged}). 
We therefore introduce Sparse \SAGA, a novel variant which explicitly takes sparsity into account and is easily parallelizable. 

\vspace{-2mm}
\paragraph{Sparse \SAGA\ Algorithm.}
As in the Sparse \SVRG\ algorithm proposed in~\citet{mania}, we obtain Sparse \SAGA\ by a simple modification of the parameter update rule in~\eqref{eq:SAGAupdate} where $\bar{\alpha}$ is replaced by a sparse version equivalent in expectation: 
\begin{equation} \label{eq:SparseSAGA}
x^{+} = x - \stepsize (f'_i(x) - \alpha_i + D_i \bar \alpha),
\end{equation}
where $D_i$ is a diagonal matrix that makes a weighted projection on the support of $f'_i$.
More precisely, let $S_i$ be the support of the gradient $f_i'$ function (i.e., the set of coordinates where $f_i'$ can be nonzero). Let $D$ be a $d\times d$ diagonal reweighting matrix, with coefficients~$\nicefrac{1}{p_v}$ on the diagonal, where $p_v$ is the probability that dimension~$v$ belongs to $S_i$ when $i$ is sampled uniformly at random in $\{1,...,n\}$.
We then define $D_i := P_{S_i} D$, where $P_{S_i}$ is the projection onto $S_i$. The normalization from $D$
ensures that $\Econd D_i \bar \alpha = \bar{\alpha}$, and thus that the update is still unbiased despite the projection.

\vspace{-2mm}
\paragraph{Convergence Result for (Serial) Sparse \SAGA.}
For clarity of exposition, we model our convergence result after the simple form of~\citet[Corollary~3]{qsaga} (note that the rate for Sparse \SAGA\ is the same as \SAGA). The proof is given in Appendix~\ref{apxA}.

\begin{theorem}\label{th1}
Let $\stepsize = \frac{a}{5\lipschitz}$ for any $a\leq1$. 
Then \textnormal{Sparse \SAGA} converges geometrically in expectation with a rate factor of at least $\rho(a) = \frac{1}{5} \min\big\{\frac{1}{n}, a\frac{1}{\kappa}\big\}$, i.e., for $x_t$ obtained after $t$ updates, we have ${\E \|x_t - x^*\|^2} \leq {(1-\rho)}^t \,  C_0$, where $C_0 := \|x_0 - x^*\|^2 + \frac{1}{5L^2} \sum_{i=1}^n {\|\alpha_i^0 - f'_i(x^*)\|^2}$.
\end{theorem}

\vspace{-2mm}
\paragraph{Comparison with Lagged Updates.}
The lagged updates technique in \SAGA\ is based on the observation that the updates for component $[x]_v$ can be delayed until this coefficient is next accessed. 
Interestingly, the expected number of iterations between two steps where a given dimension $v$ is involved in the partial gradient is $p_v^{-1}$, where $p_v$ is the probability that $v$ is involved. 
$p_v^{-1}$ is precisely the term which we use to multiply the update to $[x]_v$ in Sparse \SAGA. 
Therefore one may view the Sparse \SAGA\ updates as \textit{anticipated} \SAGA\ updates, whereas those in the~\citet{laggedsaga} implementation are \textit{lagged}. 

Although Sparse \SAGA\ requires the computation of the $p_v$ probabilities, this can be done during a first pass through the data (during which constant step size \SGD\ may be used) at a negligible cost. 
In our experiments, both Sparse \SAGA\ and \SAGA\ with lagged updates had similar convergence in terms of number of iterations, with the Sparse \SAGA\ scheme being slightly faster in terms of runtime. 
We refer the reader to~\citet{laggedsaga} and Appendix~\ref{apxC} for more details.

\section{Asynchronous Parallel Sparse \SAGA} \label{sec:ASAGA}
\vspace{-2mm}
As most recent parallel optimization contributions, we use a similar hardware model to~\citet{hogwild}. 
We have multiple cores which all have read and write access to a shared memory. 
They update a central parameter vector in an asynchronous and lock-free fashion. 
Unlike~\citet{hogwild}, we \emph{do not} assume that the vector reads are consistent: multiple cores can read and write different coordinates of the shared vector at the same time. This means that a full vector read for a core might not correspond to any consistent state in the shared memory at any specific point in time. 

\vspace{-1mm}
\subsection{Perturbed Iterate Framework}
\vspace{-2mm}
We first review the ``perturbed iterate'' framework recently introduced by~\citet{mania} which will form the basis of our analysis. 
In the sequential setting, stochastic gradient descent and its variants can be characterized by the following update rule:
\begin{equation}
x_{t+1} = x_t -\stepsize g(x_t, i_t),
\end{equation}
where $i_t$ is a random variable independent from $x_t$ and we have the unbiasedness condition $\Econd g(x_t, i_t) = f'(x_t)$ (recall that $\Econd$ is the relevant-past conditional expectation with respect to $i_t$).

Unfortunately, in the parallel setting, we manipulate stale, inconsistent reads of shared parameters and thus we do not have such a straightforward relationship. 
Instead, \citet{mania} proposed to separate $\hat x_t$, the actual value read by a core to compute an update, with $x_t$, a ``virtual iterate'' that we can analyze and is \emph{defined} by the update equation:
$
x_{t+1} := x_t -\stepsize g(\hat x_t, i_t).
$ We can thus interpret $\hat x_t$ as a noisy (perturbed) version of $x_t$ due to the effect of asynchrony. In the specific case of (Sparse) \SAGA, we have to add the additional read memory argument $\hat{\alpha}^t$ to our update:
\begin{equation}  \label{eq:PIupdate}
\begin{aligned}
x_{t+1} &:= x_t -\stepsize g(\hat x_t, \hat \alpha^t, i_t);
\\
g(\hat x_t, \hat \alpha^t, i_t) &:= f'_{i_t}(\hat x_t) - \hat \alpha_{i_t}^t + D_{i_t} \left({\textstyle \nicefrac{1}{n} \sum_{i=1}^n \hat{\alpha}_i^t }\right).
\end{aligned}
\end{equation}
We formalize the precise meaning of $x_t$ and $\hat x_t$ in the next section. 
We first note that all the papers mentioned in the related work section that analyzed asynchronous parallel randomized algorithms assumed that the following unbiasedness condition holds:
\begin{equation} \label{eq:unbiasedness}
\left[{\text{unbiasedness} \atop \text{condition}}\right] \quad \Econd [g(\hat x_{t}, i_t) | \hat x_t] = f'(\hat x_t).
\end{equation}
This condition is at the heart of most convergence proofs for randomized optimization methods.\footnote{A notable exception is \SAG~\citep{SAG} which has biased updates, yielding a significantly more complex convergence proof. Making \SAG\ unbiased leads to \SAGA~\citep{SAGA} and a much simpler proof.}
\citet{mania} correctly pointed out that most of the literature thus made the often implicit assumption that $i_t$ is independent of $\hat{x}_t$. 
But as we explain below, this assumption is incompatible with a non-uniform asynchronous model in the analysis approach used in most of the recent literature.

\vspace{-1mm}
\subsection{On the Difficulty of Labeling the Iterates} \label{ssec:labelingIssue}
\vspace{-2mm}
Formalizing the meaning of $x_t$ and $\hat x_t$ highlights a subtle but important difficulty arising when analyzing \emph{randomized} parallel algorithms: what is the meaning of $t$? 
This is the problem of \emph{labeling} the iterates for the purpose of the analysis, and this labeling can have randomness itself that needs to be taken in consideration when interpreting the meaning of an expression like $\E[x_t]$. 
In this section, we contrast three different approaches in a unified framework. 
We notably clarify the dependency issues that the labeling from~\citet{mania} resolves and propose a new, simpler labeling which allows for much simpler proof techniques.
We consider algorithms that execute in parallel the following four steps, where $t$ is a global labeling that needs to be defined:
\begin{equation} \label{eq:updates}
\begin{minipage}{0.43\textwidth}
\begin{enumerate}
\setlength\itemsep{-0.2em}
\item Read the information in shared memory ($\hat{x}_t$).
\item Sample $i_t$.
\item Perform some computations using ($\hat{x}_t, i_t$).
\item Write an update to shared memory.
\end{enumerate}
\end{minipage}
\end{equation}

\vspace{-4mm}
\paragraph{The ``After Write'' Approach.} We call the ``after write'' approach the standard global labeling scheme used in~\citet{hogwild} and re-used in all the later papers that we mentioned in the related work section, with the notable exceptions of~\citet{mania} and~\citet{duchi}. In this approach, $t$ is a (virtual) global counter recording the number of \emph{successful writes} to the shared memory $x$ (incremented after step~4 in~\eqref{eq:updates}); $x_t$ thus represents the (true) content of the shared memory after $t$ updates. 
The interpretation of the crucial equation~\eqref{eq:PIupdate} then means that $\hat{x}_t$ represents the (delayed) local copy value of the core that made the $(t+1)^{\mathrm{th}}$ successful update; $i_t$ represents the factor sampled by this core for this update. 
Notice that in this framework, the value of $\hat x_t$ and $i_t$ is unknown at ``time~$t$''; we have to wait to the later time when the next core writes to memory to finally determine that its local variables are the ones labeled by $t$.
We thus see that here~$\hat x_t$ and~$i_t$ are not necessarily independent -- they share dependence through the $t$ label assignment. 
In particular, if some values of~$i_t$ yield faster updates than others, it will influence the label assignment defining $\hat x_t$. We illustrate this point with a concrete problematic example in Appendix~\ref{apx:ProblematicExample} that shows that in order to preserve the unbiasedness condition~\eqref{eq:unbiasedness}, the ``after write'' framework makes the implicit assumption that the computation time for the algorithm running on a core is independent of the sample $i$ chosen.
This assumption seems overly strong in the context of potentially heterogeneous factors $f_i$'s, and is thus a fundamental flaw for analyzing non-uniform asynchronous computation.

\vspace{-2mm}
\paragraph{The ``Before Read'' Approach.}
\citet{mania} addresses this issue by proposing instead to increment the global~$t$~counter just \emph{before} a new core starts to \emph{read} the shared memory (before step~1 in~\eqref{eq:updates}).  
In their framework, $\hat{x}_t$ represents the (inconsistent) read that was made by this core in this computational block, and $i_t$ represents the picked sample. 
The update rule~\eqref{eq:PIupdate} represents a \emph{definition} of the meaning of $x_t$, which is now a ``virtual iterate'' as it does not necessarily correspond to the content of the shared memory at any point. 
The real quantities manipulated by the algorithm in this approach are the $\hat{x}_t$'s, whereas $x_t$ is used only for the analysis -- the critical quantity we want to see vanish is $\E \|\hat x_t - x^*\|^2$.
The independence of~$i_t$ with~$\hat{x}_t$ can be simply enforced in this approach by making sure that the way the shared memory $x$ is read does not depend on $i_t$ (e.g. by reading all its coordinates in a fixed order). Note that this means that we have to read all of $x$'s coordinates, regardless of the size of $f_{i_t}$'s support. 
This is a much weaker condition than the assumption that all the computation in a block does not depend on $i_t$ as required by the ``after write'' approach, and is thus more reasonable.

\vspace{-2mm}
\paragraph{A New Global Ordering: the ``After Read'' Approach.}
The ``before read'' approach gives rise to the following complication in the analysis: $\hat{x}_t$ can depend on $i_r$ for $r > t$. 
This is because $t$ is a global time ordering only on the assignment of computation to a core, not on when  $\hat{x}_t$ was finished to be read. 
This means that we need to consider both the ``future'' and the ``past'' when analyzing $x_t$. 
To simplify the analysis (which proved crucial for our \ASAGA\ proof), we thus propose a third way to label the iterates: $\hat{x}_t$ represents the $(t+1)^{\mathrm{th}}$ \emph{fully completed read} ($t$ incremented after step~1 in~\eqref{eq:updates}). 
As in the ``before read'' approach, we can ensure that $i_t$ is independent of $\hat{x}_t$ by ensuring that how we read does not depend on $i_t$. 
But unlike in the ``before read'' approach, $t$ here now does represent a global ordering on the $\hat{x}_t$ iterates -- and thus we have that $i_r$ is independent of $\hat{x}_t$ for $r > t$. 
Again using~\eqref{eq:PIupdate} as the definition of the virtual iterate $x_t$ as in the perturbed iterate framework, we then have a very simple form for the value of $x_t$ and $\hat{x}_t$ (assuming atomic writes, see Property~\ref{eventconst} below):
\begin{equation} \label{eq:xhatUpdates}
\begin{aligned}
x_t &= x_0 - \stepsize \sum_{u = 0}^{t-1} g(\hat x_u, \hat \alpha^u, i_u) \, ;
\\
[\hat{x}_t]_v &= [x_0]_v - \stepsize  
\mkern-36mu \sum_{\substack{u = 0 \\ 
		\text{u s.t. coordinate $v$ was written}\\ 
		\text{for $u$ before $t$} } 
}^{t-1} \mkern-36mu  [g(\hat x_u, \hat \alpha^u, i_u)]_v \, .
\end{aligned}
\end{equation}
The main idea of the perturbed iterate framework is to use this handle on $\hat x_t - x_t$ to analyze the convergence for $x_t$. 
In this paper, we can instead give directly the convergence of $\hat x_t$, and so unlike in~\citet{mania}, we do not require that there exists a $T$ such that $x_T$ lives in shared memory.

\begin{figure*}[ttt!]
	\vspace{-5mm}
 \begin{minipage}[t]{0.49\textwidth}
   \begin{algorithm}[H]
     \caption{\ASAGA\ (analyzed algorithm)}
     \label{alg:theoretical}
     \label{theoreticalgo}
     \begin{algorithmic}[1]
	   \STATE Initialize shared variables $x$ and $(\alpha_i)_{i=1}^n$
	   \LOOP
	      \STATE $\hat x = $ inconsistent read of $x$
		  \STATE $\forall j$, $\hat \alpha_j = $ inconsistent read of $\alpha_j$
	   	  \STATE \textcolor{blue}{Sample $i$}  uniformly at random in $\{1,...,n\}$
	      \STATE Let $S_i$ be $f_i$'s support
	      \STATE $[\bar \alpha]_{S_i} = \nicefrac{1}{n} \sum_{k=1}^n [\hat \alpha_k]_{S_i}$
    		  \STATE $[\delta x]_{S_i} = -\stepsize (f'_i(\hat x) - \hat \alpha_i + D_{i} [\bar \alpha]_{S_i})$
    		  \STATE
        	  \FOR{$v$ {\bfseries in} $S_i$}
        		 \STATE $[x]_v \leftarrow [x]_v + [\delta x]_v$      \hfill // atomic
	         \STATE $[\alpha_i]_v \leftarrow [f'_i( \hat x )]_v$
	         \STATE // {\small(`$\gets$' denotes a shared memory update.)}
    		  \ENDFOR
	   \ENDLOOP
	  \end{algorithmic}
    \end{algorithm}
 \end{minipage}
 \hfill
 \begin{minipage}[t]{0.5\textwidth}
    \begin{algorithm}[H]
      \caption{\ASAGA\ (implementation)}
      \label{alg:sagasync}
      \begin{algorithmic}[1]
	    \STATE Initialize shared variables $x$, $(\alpha_i)_{i=1}^n$ and $\bar \alpha$
	    \LOOP
	      \STATE \textcolor{blue}{Sample $i$} uniformly at random in $\{1,...,n\}$
  	      \STATE Let $S_i$ be $f_i$'s support
 	      \STATE $[\hat x]_{S_i} = $ inconsistent read of $x$ on $S_i$
	      \STATE $\hat \alpha_i = $ inconsistent read of $\alpha_i$
	      \STATE $[\bar \alpha]_{S_i} = $ inconsistent read of $\bar \alpha$ on $S_i$
	      \STATE $[\delta \alpha]_{S_i} = f'_i([\hat x]_{S_i}) - \hat \alpha_i$
	      \STATE $[\delta x]_{S_i} = - \gamma ([\delta\alpha]_{S_i} + D_i [\bar \alpha]_{S_i})$
	      \FOR{$v$ {\bfseries in} $S_i$}
	        \STATE $[x]_v \gets [x]_v + [\delta x]_v$  \hfill // atomic
	        \STATE $[\alpha_i]_v \gets [\alpha_i]_v + [\delta \alpha]_v$ \hfill // atomic
	        \STATE $[\bar \alpha]_v \gets [\bar \alpha]_v + \nicefrac{1}{n}[\delta \alpha]_v$ \hfill // atomic
	     \ENDFOR
	   \ENDLOOP
      \end{algorithmic}
    \end{algorithm}
 \end{minipage}
 	\vspace{-5mm}
\end{figure*}

\vspace{-1mm}
\subsection{Analysis setup} \label{ssec:convergence}
\vspace{-2mm}
We describe \ASAGA, a sparse asynchronous parallel implementation of Sparse \SAGA, in Algorithm~\ref{alg:theoretical} in the theoretical form that we analyze, and in Algorithm~\ref{alg:sagasync} as its practical implementation.
Before stating its convergence, we highlight some properties of Algorithm~\ref{alg:theoretical} and make one central assumption. 

\begin{prop} [independence]
Given the ``after read'' global ordering, $i_r$ is independent of $\hat{x}_t$ $\forall r \geq t$.
\label{independence}
\end{prop}
\vspace{-2mm}
We enforce the independence for $r = t$ in Algorithm~\ref{alg:theoretical} by having the core read all the shared data parameters and historical gradients before starting their iterations. 
Although this is too expensive to be practical if the data is sparse, this is required by the theoretical Algorithm~\ref{theoreticalgo} that we can analyze. 
As~\citet{mania} stress, this independence property is assumed in most of the parallel optimization literature. The independence for $r > t$ is a consequence of using the ``after read'' global ordering instead of the ``before read'' one.

\begin{prop}[Unbiased estimator]
The update, $g_t := g(\hat x_t, \hat \alpha^t, i_t)$, is an unbiased estimator of the true gradient at $\hat x_t$ (i.e.~\eqref{eq:PIupdate} yields~\eqref{eq:unbiasedness} in conditional expectation).
\end{prop}
\vspace{-2mm}
This property is crucial for the analysis, as in most related literature. 
It follows by the independence of $i_t$ with $\hat{x}_t$ and from the computation of $\bar \alpha$ on line~7 of Algorithm~\ref{theoreticalgo}, which ensures that $\E \hat \alpha_i = 1/n \sum_{k=1}^n [\hat \alpha_k]_{S_i} = [\bar \alpha]_{S_i}$, making the update unbiased. In practice, recomputing $\bar \alpha$ is not optimal, but storing it instead introduces potential bias issues in the proof (as detailed in Appendix~\ref{apx:Bias}).

\begin{prop} [atomicity]
The shared parameter coordinate update of $[x]_v$ on line~11 is atomic.
\label{eventconst}
\end{prop}
\vspace{-2mm}
Since our updates are additions, this means that there are no overwrites, even when several cores compete for the same resources. 
In practice, this is enforced by using \textit{compare-and-swap} semantics, which are heavily optimized at the processor level and have minimal overhead. 
Our experiments with non-thread safe algorithms (i.e. where this property is not verified, see Figure~\ref{fig:cas_comparison} of Appendix~\ref{apxE}) show that compare-and-swap is necessary to optimize to high accuracy.

Finally, as is standard in the literature, we make an assumption on the maximum delay that asynchrony can cause -- this is the \emph{partially asynchronous} setting as defined in~\citet{bertsekasParalle1989}:
\begin{asm}[bounded overlaps]\label{boundedoverlap}
We assume that there exists a uniform bound, called~$\overlap$, on the maximum number of iterations that can overlap together. We say that iterations $r$ and $t$ overlap if at some point they are processed concurrently. 
One iteration is being processed from the start of the reading of the shared parameters to the end of the writing of its update. 
The bound~$\overlap$ means that iterations $r$ cannot overlap with iteration $t$ for $r \geq t + \tau+1$,  and thus that every coordinate update from iteration $t$ is successfully written to memory before the iteration $t + \overlap+1$ starts.
\label{boundedoverlaps}
\end{asm}

\vspace{-2mm}
Our result will give us conditions on $\overlap$ subject to which we have linear speedups. 
$\overlap$ is usually seen as a proxy for $p$, the number of cores (which lowerbounds it). 
However, though $\overlap$ appears to depend linearly on $p$, it actually depends on several other factors (notably the data sparsity distribution) and can be orders of magnitude bigger than $p$ in real-life experiments.
We can upper bound $\overlap$ by $(p-1)R$, where $R$ is the ratio of the maximum over the minimum iteration time (which encompasses theoretical aspects as well as hardware overhead).
More details can be found in Appendix~\ref{apxD}.

\vspace{-2mm}
\paragraph{Explicit effect of asynchrony.}
By using the overlap Assumption~\ref{boundedoverlaps} in the expression~\eqref{eq:xhatUpdates} for the iterates, we obtain the following explicit effect of asynchrony that is crucially used in our proof:
\begin{align}\label{eq:async}
\hat x_t - x_t = \stepsize \sum_{u=(t - \overlap)_+}^{t-1}G_{u}^t g(\hat x_{u}, \hat \alpha^u, i_{u}),
\end{align}
where $G_{u}^t$ are $d\times d$ diagonal matrices with terms in $\{0, +1\}$. 
We know from our definition of~$t$ and~$x_t$ that every update in $\hat x_t$ is already in $x_t$ -- this is the $0$ case. 
Conversely, some updates might be late: this is the $+1$ case. 
$\hat x_t$ may be lacking some updates from the ``past" in some sense, whereas given our global ordering definition, it cannot contain updates from the ``future".

\vspace{-1mm}
\subsection{Convergence and speedup results}
\vspace{-2mm}
We now state our main theoretical results. We give an outline of the proof in Section~\ref{proofoutline} and its full details in Appendix~\ref{apxB}.
We first define a notion of problem sparsity, as it will appear in our results.

\begin{definition}[Sparsity]
As in~\citet{hogwild}, we introduce $\sparsityr := \max_{v=1..d} |\{i : v \in S_i\}|$. $\sparsityr$ is 
the maximum right-degree in the bipartite graph of the factors and the dimensions, i.e., 
the maximum number of data points with a specific feature. For succinctness, we also define $\sparsity := \sparsityr / n$. We have $1 \leq \sparsityr \leq n$, and hence $1/n \leq \sparsity \leq 1$.

\end{definition}

\begin{theorem}[Convergence guarantee and rate of \ASAGA]\label{thm:convergence}
Suppose $\overlap < n/10$.\footnote{\ASAGA\ can actually converge for any $\overlap$, but the maximum step size then has a term of $\exp(\overlap/n)$ in the denominator with much worse constants. See Appendix~\ref{apxB:lma3}.} Let
\vspace{-1mm}
\begin{equation}\label{eq:condition}
\begin{aligned}
&a^*(\overlap) := \frac{1}{32 \left(1+ \overlap  \sqrt \sparsity \right) \xi(\kappa, \sparsity, \overlap)}
\\
&\text{where } \xi(\kappa, \sparsity, \overlap) := \sqrt{1 + \frac{1}{8 \kappa}  \min\{\frac{1}{\sqrt{\sparsity}}, \overlap\} } \\
&\text{\small{(note that $\xi(\kappa, \sparsity, \overlap) \approx 1$ unless $\kappa < \nicefrac{1}{\sqrt{\sparsity}}  \,\, (\leq \sqrt{n})$)}}.
\end{aligned}
\end{equation}
For any step size $\stepsize = \frac{a}{L}$ with $a \leq a^*(\overlap)$, the inconsistent read iterates of Algorithm~\ref{alg:theoretical} converge in expectation at a geometric rate of at least: $\contraction(a) = \frac{1}{5} \min \big\{\frac{1}{n},  a \frac{1}{\kappa}\big\},$
i.e., $\E f(\hat x_t)-f(x^*) \leq (1-\rho)^t \,  \tilde C_0$, where $\tilde C_0$ is a constant independent of $t$  ($\approx \frac{n}{\stepsize}C_0$ with $C_0$ as defined in Theorem~\ref{th1}).
\end{theorem}
\vspace{-2mm}
This result is very close to \SAGA's original convergence theorem, but with the maximum step size divided by an extra $1+ \overlap \sqrt{\sparsity}$ factor. Referring to~\citet{qsaga} and our own Theorem~\ref{th1}, the rate factor for \SAGA\ is $\min\{1/n, 1/\kappa\}$ up to a constant factor. Comparing this rate with Theorem~\ref{thm:convergence} and inferring the conditions on the maximum step size $a^*(\overlap)$, we get the following conditions on the overlap $\overlap$ for \ASAGA\ to have the same rate as \SAGA\ (comparing upper bounds).
\begin{corollary}[Speedup condition]\label{thm:bigdata}\label{thm:illcondition}
Suppose $\overlap \leq \mathcal{O}(n)$ and $\overlap \leq \mathcal{O}({\scriptstyle \frac{1}{\sqrt{\sparsity}}} \max\{1,\frac{n}{\kappa} \})$. Then using the step size $\stepsize = \nicefrac{a^*(\overlap)}{L}\,$ from~\eqref{eq:condition}, \ASAGA\ converges geometrically with rate factor $\Omega( \min\{\frac{1}{n}, \frac{1}{\kappa}\})$ (similar to \SAGA), and is thus linearly faster than its sequential counterpart up to a constant factor. Moreover, if $\overlap \leq \mathcal{O}(\frac{1}{\sqrt{\sparsity}})$, then a universal step size of $\Theta(\frac{1}{L})$ can be used for \ASAGA\ to be adaptive to local strong convexity with a similar rate to \SAGA\ (i.e., knowledge of $\kappa$ is not required).
\end{corollary}

\vspace{-2mm}
Interestingly, in the well-conditioned regime ($n > \kappa$, where \SAGA\ enjoys a range of stepsizes which all give the same contraction ratio), \ASAGA\ can get the same rate as \SAGA\ even in the non-sparse regime ($\sparsity = 1$) for $\overlap < \mathcal{O}(n/\kappa)$. This is in contrast to the previous work on asynchronous incremental gradient methods which required some kind of sparsity to get a theoretical linear speedup over their sequential counterpart~\citep{hogwild,mania}. 
In the ill-conditioned regime ($\kappa > n$), sparsity is required for a linear speedup, with a bound on $\overlap$ of $\mathcal{O}(\sqrt{n})$ in the best-case (though degenerate) scenario where $\sparsity = 1/n$.

\vspace{-2mm}
\paragraph{Comparison to related work.}
\begin{itemize}[leftmargin=1.5em, topsep=-1mm, itemsep=-0.7mm]
\item We give the first convergence analysis for an asynchronous parallel version of \SAGA\ (note that \citet{smola} only covers an epoch based version of \SAGA\ with random stopping times, a fairly different algorithm).
\item Theorem~\ref{thm:convergence} can be directly extended to a parallel extension of the \SVRG\ version from~\citet{qsaga}, which is adaptive to the local strong convexity with similar rates (see Appendix~\ref{apx:SVRGext}).
\item In contrast to the parallel \SVRG\ analysis from~\citet[Thm. 2]{smola}, our proof technique handles inconsistent reads and a non-uniform processing speed across $f_i$'s. 
Our bounds are similar (noting that $\sparsity$ is equivalent to theirs), except for the adaptivity to local strong convexity: \ASAGA\ does not need to know $\kappa$ for optimal performance, contrary to parallel \SVRG\ (see App.~\ref{apx:SVRGext} for more details).
\item In contrast to the \SVRG\ analysis from~\citet[Thm. 14]{mania}, we obtain a better dependence on the condition number in our rate ($1/\kappa$ vs. $1/\kappa^2$ for them) and on the sparsity (they get $\overlap \leq \mathcal{O}(\sparsity^{\nicefrac{-1}{3}})$), while we remove their gradient bound assumption.
We also give our convergence guarantee on $\hat{x}_t$ \emph{during} the algorithm, whereas they only bound the error for the ``last'' iterate $x_T$.
\end{itemize}

\vspace{-2mm}
\subsection{Proof outline}\label{proofoutline}
\vspace{-2mm}
We give here the outline of our proof. Its full details can be found in Appendix~\ref{apxB}.

Let $g_t := g(\hat x_t, \hat \alpha^t, i_t)$. By expanding the update equation~\eqref{eq:PIupdate} defining the virtual iterate $x_{t+1}$ and introducing $\hat x_t$ in the inner product term, we get:
\begin{equation}\label{eq:initrec0}
\begin{aligned}
\|x_{t+1} - x^*\|^2
&= \|x_t -x^*\|^2 
-2\stepsize\langle \hat x_t -x^*,  g_t\rangle 
\\
&\quad+2\stepsize\langle \hat x_t -x_t,  g_t\rangle 
+ \stepsize^2 \|g_t\|^2 .
\end{aligned}
\end{equation}
In the sequential setting, we require $i_t$ to be independent of $x_t$ to get unbiasedness. 
In the perturbed iterate framework, we instead require that $i_t$ is independent of $\hat x_t$ (see Property~\ref{independence}). 
This crucial property enables us to use the unbiasedness condition~\eqref{eq:unbiasedness} to write:
$\E \langle \hat x_t -x^*,  g_t\rangle 
= \E \langle \hat x_t -x^*,  f'(\hat x_t)\rangle$. We thus take the expectation of~\eqref{eq:initrec0} that allows us to use the $\mu$-strong convexity of $f$:\footnote{Here is our departure point with \citet{mania} who replaced the $f(\hat{x}_t)-f(x^*)$ term with the lower bound $\frac{\strongconvex}{2}\|\hat x_t - x^*\|^2$ in this relationship (see their Equation (2.4)), yielding an inequality too loose to get fast rates for \SVRG.} 
\begin{align}
\langle \hat x_t -x^*,  f'(\hat x_t)\rangle &\geq f(\hat x_t) -f(x^*) +\frac{\strongconvex}{2}\|\hat x_t - x^*\|^2 . \notag
\end{align}
With further manipulations on the expectation of~\eqref{eq:initrec0}, including the use of the standard inequality $\|a + b\|^2 \leq 2\|a\|^2 + 2\|b\|^2$ (see Section~\ref{app:RecursiveDerivation}), we obtain our basic recursive contraction inequality:
\begin{equation}
\begin{aligned} \label{eq:RecursiveIneq1}
a_{t+1} \leq 
&(1 -\frac{\stepsize \strongconvex}{2}) a_t 
+ \stepsize^2 \E \|g_t\|^2 -2\stepsize e_t
\\
&\underbrace{
	+ \stepsize\strongconvex \E\|\hat x_t - x_t\|^2 
	+ 2\stepsize \E \langle \hat x_t -x_t,  g_t\rangle
}_{\text{additional asynchrony terms}}\,,
\end{aligned}
\end{equation}
where $a_t := \E \|x_t - x^*\|^2$ and $e_t := \E f(\hat x_t) - f(x^*)$.

In the sequential setting, one crucially uses the negative suboptimality term $-2\stepsize e_t$ to cancel the variance term $\stepsize^2 \E \|g_t\|^2$ (thus deriving a condition on $\stepsize$).
Here, we need to bound the additional asynchrony terms using the same negative suboptimality in order to prove convergence and speedup for our parallel algorithm -- thus getting stronger constraints on the maximum step size.

The rest of the proof then proceeds as follows:
\vspace{-2mm}
\begin{itemize}[leftmargin=1.5em, topsep=1mm, itemsep=0mm]
	\item Lemma~\ref{lma:1}: we first bound the additional asynchrony terms in~\eqref{eq:RecursiveIneq1} in terms of past updates ($\E \|g_u\|^2, u\leq t$). 
	We achieve this by crucially using the expansion~\eqref{eq:async} for $x_t - \hat{x}_t$, together with the sparsity inequality~\eqref{sparseproduct} (which is derived from Cauchy-Schwartz, see Appendix~\ref{apxB:lma1}).
	
	\item Lemma~\ref{lma:suboptgt}: we then bound the updates $\E \|g_u\|^2$ with respect to past suboptimalities $(e_v)_{v\leq u}$. From our analysis of Sparse \SAGA\ in the sequential case:
	\begin{equation}
	\E\|g_t\|\sp{2} 
	\leq 2 \E \|f'_{i_t}(\hat x_t)-f'_{i_t}(x\sp{*})\|^2 
	+ 2 \E \|\hat \alpha_{i_t}\sp{t} - f'_{i_t}(x\sp{*})\|^2 \notag
	\end{equation}
	We bound the first term by $4Le_t$ using~\citet[Equation (8)]{qsaga}.
	To express the second term in terms of past suboptimalities, we note that it can be seen as an expectation of past first terms with an adequate probability distribution which we derive and bound.
	
	\item By substituting Lemma~\ref{lma:suboptgt} into Lemma~\ref{lma:1}, we get a master contraction inequality~\eqref{master} in terms of $a_{t+1}$, $a_t$ and $e_u, u\leq t$.
	
	\item We define a novel Lyapunov function $\lyapunov_t = \sum_{u=0}^t (1-\contraction)^{t-u}a_u$ and manipulate the master inequality to show that $\lyapunov_t$ is bounded by a contraction, subject to a maximum step size condition on~$\stepsize$ (given in Lemma~\ref{lma:3}, see Appendix~\ref{apxB:outline}).
	
	\item Finally, we unroll the Lyapunov inequality to get the convergence Theorem~\ref{thm:convergence}.
\end{itemize}

\section{Empirical results}\label{sec:results}
\begin{figure*}[ttt!]
	\centering
	\begin{subfigure}[t]{0.48\linewidth}
		\centering
		\includegraphics[width = 1.05\linewidth]{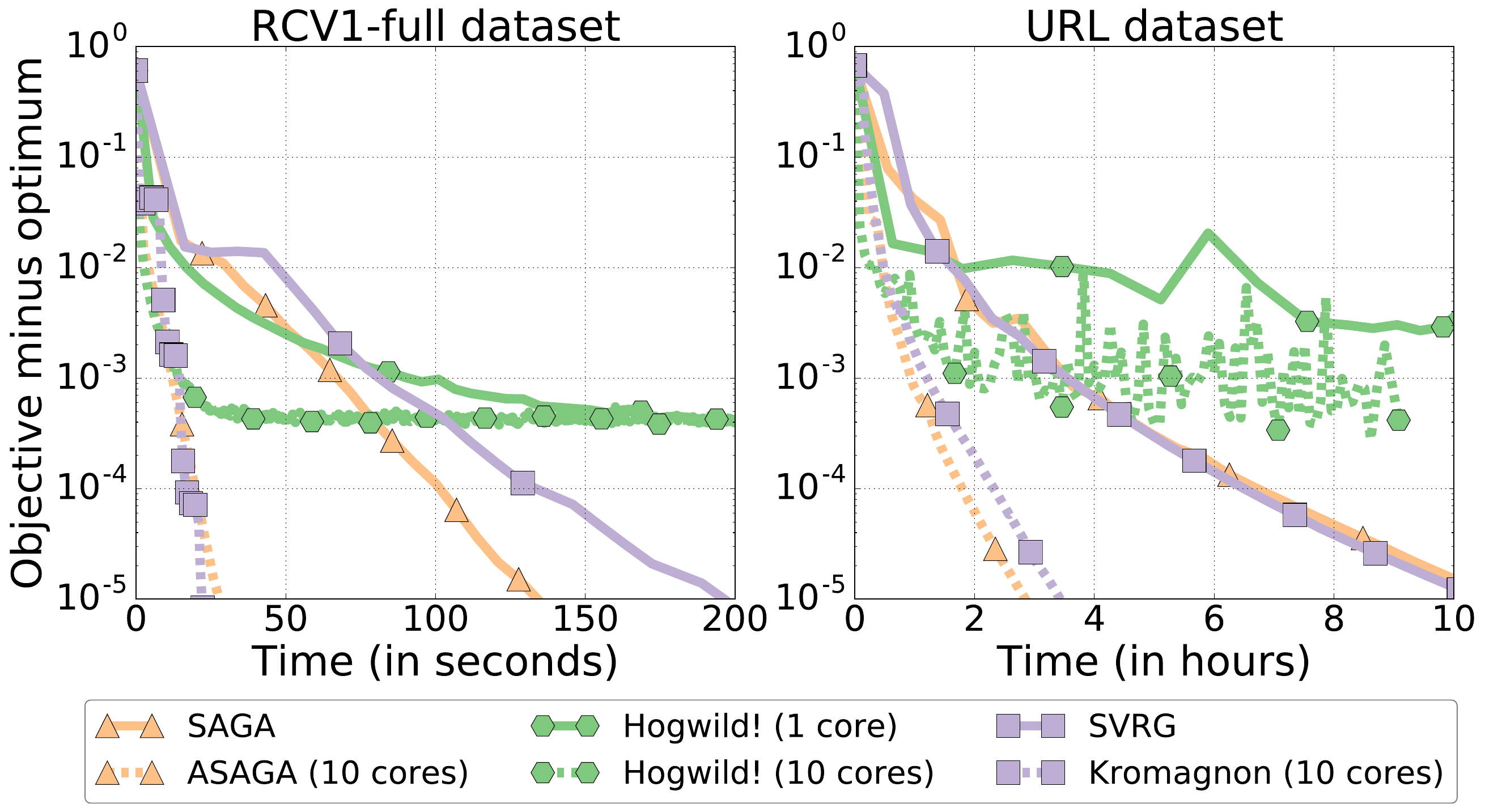}
		\caption{Suboptimality as a function of time.}
		\label{fig:fig_2}
	\end{subfigure}
	\hfill
	\begin{subfigure}[t]{0.48\linewidth}
		\centering
		\includegraphics[width = 1\linewidth]{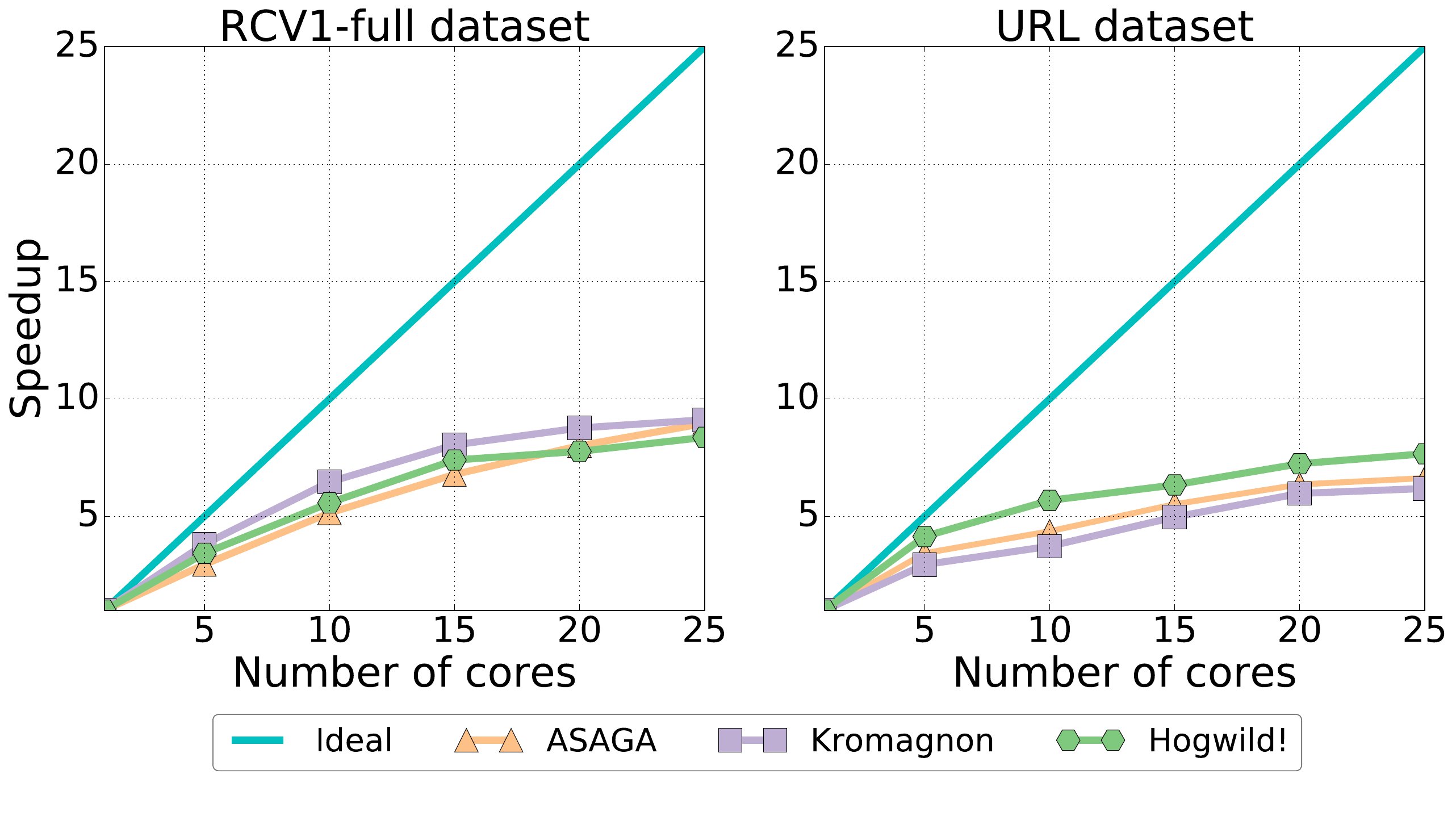}
		\caption{Speedup as a function of the number of cores}
		\label{fig:fig_3}
	\end{subfigure}
	\caption{ {\bf Convergence and speedup for asynchronous stochastic gradient descent methods}.
		We display results for RCV1 and URL. Results for Covtype can be found in Appendix~\ref{apx:speedup}. } 
	\vspace{-3.5mm}
\end{figure*}
\vspace{-2mm}
We now present the main results of our empirical comparison of asynchronous \SAGA , \SVRG\ and \Hogwild. 
Additional results, including convergence and speedup figures with respect to the number of iteration and measures on the $\overlap$ constant are available in the appendix.
\vspace{-6mm}
\subsection{Experimental setup}\label{ImplDetails}
\vspace{-2mm}
{\bf Models.} 
Although \ASAGA\ can be applied more broadly, we focus on logistic regression, a model of particular practical importance. 
The associated objective function takes the following form:
$
{\frac{1}{n} \sum_{i=1}^n \log\big(1 + \exp(- b_i a_i^\intercal x)\big)} + \frac{\lambda}{2} \|x\|^2,
$
where $a_i \in \mathbb{R}^p$ and $b_i \in \{-1,+1\}$ are the data samples.

{\bf Datasets.}
We consider two sparse datasets: RCV1~\citep{RCV1} and URL~\citep{URL}; and a dense one, Covtype~\citep{Covtype}, with statistics listed in the table below. 
As in~\citet{SAG}, Covtype is standardized, thus $100\%$ dense.  
$\sparsity$ is $\mathcal{O}(1)$ in all datasets, hence not very insightful when relating it to our theoretical results.
Deriving a less coarse sparsity bound remains an open problem.

\begin{table}[h]
\resizebox{0.48\textwidth}{!}{
\begin{tabular}{lcccc}
\toprule
{} & $n$ & $d$ & density & $\lipschitz$\\
\midrule
{\bf RCV1} & \hfill 697,641 & \hfill 47,236 & \hfill 0.15\% & \hfill 0.25\\ 
{\bf URL} & \hfill 2,396,130 & \hfill 3,231,961 & \hfill 0.004\% & \hfill 128.4\\
{\bf Covtype} & \hfill 581,012 & \hfill 54 & \hfill 100\% & \hfill 48428\\
\bottomrule
\end{tabular}
}
\end{table}

\vspace{-1mm}
{\bf Hardware and software}. 
Experiments were run on a 40-core machine with 384GB of memory. 
All algorithms were implemented in Scala. We chose this high-level language despite its typical 20x slowdown compared to C (when using standard libraries, see Appendix~\ref{scalavsc}) because our primary concern was that the code may easily be reused and extended for research purposes (to this end, we have made all our code available at \url{https://github.com/RemiLeblond/ASAGA}).

\vspace{-1mm}
\subsection{Implementation details}
\vspace{-2mm}
{\bf Exact regularization.} 
Following~\citet{laggedsaga}, the amount of regularization used was set to $\lambda=1 / n$.
In each update, we project the gradient of the regularization term (we multiply it by $D_i$ as we also do with the vector $\bar{\alpha}$) to preserve the sparsity pattern while maintaining an unbiased estimate of the gradient.
For squared $\ell_2$, the Sparse \SAGA\ updates becomes:
$
x^+ = x - \gamma (f_i'(x) - \alpha_i + D_i \bar{\alpha} + \lambda D_i x).
$

{\bf Comparison with the theoretical algorithm.} 
The algorithm we used in the experiments is fully detailed in Algorithm~\ref{alg:sagasync}. 
There are two differences with Algorithm~\ref{alg:theoretical}. 
First, in the implementation we pick $i_t$ at random \textit{before} we read data. 
This enables us to only read the necessary data for a given iteration (i.e.~$[\hat x_t]_{S_i}, [\hat \alpha_i^t], [\bar \alpha^t]_{S_i}$). 
Although this violates Property~\ref{independence}, it still performs well in practice. 

Second, we maintain $\bar \alpha^t$ in memory. 
This saves the cost of recomputing it at every iteration (which we can no longer do since we only read a subset data). 
Again, in practice the implemented algorithm enjoys good performance. 
But this design choice raises a subtle point: the update is not guaranteed to be unbiased in this setup (see Appendix~\ref{apx:Bias} for more details). 

\vspace{-2mm}
\subsection{Results}\label{ssec:results}
\vspace{-2mm}
We first compare three different asynchronous variants of stochastic gradient methods on the aforementioned datasets: \ASAGA, presented in this work, \KROMAGNON, the asynchronous sparse \SVRG\ method described in~\citet{mania} and \Hogwild~\citep{hogwild}.
Each method had its step size chosen so as to give the fastest convergence (up to $10^{-3}$ in the special case of \Hogwild). 
The results can be seen in Figure~\ref{fig:fig_2}: for each method we consider its asynchronous version with both one (hence sequential) and ten processors. 
This figure reveals that the asynchronous version offers a significant speedup over its sequential counterpart.

We then examine the speedup relative to the increase in the number of cores. 
The speedup is measured as time to achieve a suboptimality of $10^{-5}$ ($10^{-3}$ for \Hogwild) with one core divided by time to achieve the same suboptimality with several cores, averaged over 3 runs. 
Again, we choose step size leading to fastest convergence (see Appendix~\ref{scalavsc} for information about the step sizes). Results are displayed in Figure~\ref{fig:fig_3}. 

As predicted by our theory, we observe linear ``theoretical'' speedups (i.e. in terms of number of iterations, see Appendix~\ref{apx:speedup}).
However, with respect to running time, the speedups seem to taper off after $20$ cores.
This phenomenon can be explained by the fact that our hardware model is by necessity a simplification of reality. 
As noted in~\citet{duchi}, in a modern machine there is no such thing as \textit{shared memory}. 
Each core has its own levels of cache (L1, L2, L3) in addition to RAM. 
The more cores are used, the lower in the memory stack information goes and the slower it gets. 
More experimentation is needed to quantify that effect and potentially increase performance.

\vspace{-1.5mm}
\section{Conclusions and future work}
\vspace{-2.5mm}
We have described \ASAGA, a novel sparse and fully asynchronous variant of the incremental gradient algorithm \SAGA.
Building on the recently proposed ``perturbed iterate'' framework, we have introduced a novel analysis of the algorithm and proven that under mild conditions \ASAGA\ is linearly faster than \SAGA.
Our empirical benchmarks confirm speedups up to 10x.

Our proof technique accommodates more realistic settings than is usually the case in the literature (e.g. inconsistent reads/writes and an unbounded gradient); we obtain tighter conditions than in previous work. 
In particular, we show that sparsity is not always necessary to get linear speedups.
Further, we have proposed a novel perspective to clarify an important technical issue present in most of the recent convergence rate proofs for asynchronous parallel optimization algorithms.

\citet{laggedsaga} have shown that \SAG\ enjoys much improved performance when combined with non-uniform sampling and line-search.
We have also noticed that our $\sparsityr$ constant (being essentially a maximum) sometimes fails to accurately represent the full sparsity distribution of our datasets.
Finally, while our algorithm can be directly ported to a distributed master-worker architecture, its communication pattern would have to be optimized to avoid prohibitive costs. Limiting communications can be interpreted as artificially increasing the delay, yielding an interesting trade-off between delay influence and communication costs.

A final interesting direction for future analysis is the further exploration of the $\overlap$ term, which we have shown encompasses more complexity than previously thought.

\subsubsection*{Acknowledgments}
%
We would like to thank Xinghao Pan for sharing with us their implementation of \KROMAGNON, as well as Alberto Chiappa for spotting a typo in the proof.
This work was partially supported by a Google Research Award and the MSR-Inria Joint Center.
FP acknowledges financial support from
from the chaire {\em \'Economie des nouvelles donn\'ees} with the {\em data science} joint research initiative with the {\em fonds AXA pour la recherche}.

\bibliography{Asaga}
\bibliographystyle{abbrvnat}

\clearpage
\appendix
\onecolumn
\fontsize{11}{13}
\selectfont

\paragraph{Appendix Outline:}
\begin{itemize}
\item In Appendix~\ref{apx:ProblematicExample}, we give a simple example illustrating why the ``After Write'' approach can break the crucial unbiasedness condition~\eqref{eq:unbiasedness} needed for standard convergence proofs.
\item In Appendix~\ref{apxA}, we adapt the proof from~\citet{qsaga} to prove Theorem~\ref{th1}, our convergence result for serial Sparse \SAGA.
\item In Appendix~\ref{apxB}, we first give a detailed outline and then the complete details for the proof of convergence for~\ASAGA\ (Theorem~\ref{thm:convergence}) as well as its linear speedup regimes  (Corollary~\ref{thm:bigdata}).
\item In Appendix~\ref{apx:AER}, we analyze additional experimental results, including a comparison of serial \SAGA\ algorithms and a look at ``theoretical speedups'' for \ASAGA.
\item In Appendix~\ref{apxD}, we take a closer look at the $\overlap$ constant. We argue that it encompasses more complexity than is usually implied in the literature, as additional results that we present indicate.
\item In Appendix~\ref{apxC}, we compare the lagged updates implementation of \SAGA\ with our sparse algorithm, and explain why adapting the former to the asynchronous setting is difficult.
\item In Appendix~\ref{apxE}, we give additional details about the datasets and our implementation.
\end{itemize}

\section{Problematic Example for the ``After Write'' Approach} \label{apx:ProblematicExample}
We provide a concrete example to illustrate the non-independence issue arising from the ``after write'' approach. 
Suppose that we have two cores and that $f$ has two factors: $f_1$ which has support on only one variable, and $f_2$ which has support on $10^6$ variables and thus yields a gradient step that is significantly more expensive to compute. 
In the ``after write'' approach, $x_0$ is the initial content of the memory, and we do not officially know yet whether $\hat{x}_0$ is the local copy read by the first core or the second core, but we are sure that $\hat{x}_0 = x_0$ as no update can occur in shared memory without incrementing the counter. 
There are four possibilities for the next step defining $x_1$ depending on which index $i$ was sampled on each core. 
If any core samples $i=1$, we know that $x_1 = x_0 - \stepsize f'_1(x_0)$ as it will be the first (much faster update) to complete. 
This happens in 3 out of 4 possibilities; we thus have that $\E x_1 = x_0 - \stepsize (\frac{3}{4} f'_1(x_0) + \frac{1}{4} f '_2(x_0))$ -- we see that this analysis scheme \emph{does not} satisfy the crucial unbiasedness condition~\eqref{eq:unbiasedness}. 

To understand this subtle point better, note that in this very simple example, $i_0$ and $i_1$ are not independent. We can show that $P(i_1=2 \mid i_0=2) =1$. They share dependency through the labeling assignment.
 
The only way we can think to resolve this issue and ensure unbiasedness in the ``after write'' framework is to assume that the computation time for the algorithm running on a core is independent of the sample $i$ chosen.
This assumption seems overly strong in the context of potentially heterogeneous factors $f_i$'s, and is thus a fundamental flaw in the ``after write'' framework that has mostly been ignored in the recent asynchronous optimization literature. 

We note that \citet{bertsekasParalle1989} briefly discussed this issue in Section~7.8.3 of their book, stressing that their analysis for SGD required that the scheduling of computation was independent from the randomness from SGD, but they did not offer any solution if this assumption was not satisfied. Both the ``before read'' labeling from~\citet{mania} and our proposed ``after read'' labeling resolve this issue.

\clearpage

\section{Proof of Theorem~\ref{th1}}\label{apxA}
\paragraph{Proof sketch for~\citet{qsaga}.}
As we will heavily reuse the proof technique from~\citet{qsaga}, we start by giving its sketch.

First, the authors combine classical strong convexity and Lipschitz inequalities to derive the inequality~\citet[Lemma~1]{qsaga}:
\begin{align}\label{eq:sparse}
\Econd \|x^{+} \! - \!x^*\|^2 \leq 
&(1 \! - \! \stepsize\strongconvex) \|x \! -\! x^*\|^2 
+ 2\stepsize^2 \Econd \|\alpha_i - f'_i(x^*)\|^2
+ (4 \stepsize^2 \lipschitz-2\stepsize)\big(f(x) - f(x^*)\big).
\end{align}
This gives a contraction term, as well as two additional terms; $2\stepsize^2 \Econd \|\alpha_i - f'_i(x^*)\|^2$ is a positive variance term, but $(4 \stepsize^2 \lipschitz-2\stepsize)\big(f(x) - f(x^*)\big)$ is a negative suboptimality term (provided $\stepsize$ is small enough).
The suboptimality term can then be used to cancel the variance one.

Second, the authors use a classical smoothness upper bound to control the variance term and relate it to the suboptimality.
However, since the $\alpha_i$ are partial gradients computed at previous time steps, the upper bounds of the variance involve suboptimality at previous time steps, which are not directly relatable to the current suboptimality.

Third, to circumvent this issue, a Lyapunov function is defined to encompass both current and past terms.
To finish the proof,~\citet{qsaga} show that the Lyapunov function is a contraction.

\paragraph{Proof outline.}\label{sparseoutline}
Fortunately, we can reuse most of the proof from~\citet{qsaga} to show that Sparse \SAGA\ converges at the same rate as regular \SAGA.
In fact, once we establish that \citet[Lemma~1]{qsaga} is still verified we are done.

To prove this, we derive close variants of equations $(6)$ and $(9)$ in their paper, which we remind the reader of here:
\begin{align}
\Econd \|f'_i(x) - \bar \alpha_i\|^2 &\leq 2 \Econd \|f'_i(x) - f'_i(x^*)\|^2 + 2 \Econd \|\bar \alpha_i - f'_i(x^*)\|^2 \, ,
\tag*{\citet[Eq.(6)]{qsaga}}
\\
\Econd \|\bar \alpha_i - f'_i(x^*)\|^2 &\leq \Econd \|\alpha_i - f'_i(x^*)\|^2 \, .
\tag*{\citet[Eq.(9)]{qsaga}}
\end{align} 

\paragraph{Deriving~\citet[Equation (6)]{qsaga}.}
We first show that the update estimator is unbiased.
The estimator is unbiased if:
\begin{align}\label{eq:apxBias}
\Econd D_i \bar \alpha = \Econd \alpha_i = \frac{1}{n}\sum_{i=1}^n \alpha_i \, .
\end{align}

We have:
\begin{align*}
\Econd D_i \bar \alpha 
= \frac{1}{n} \sum_{i=1}^n D_i \bar \alpha
= \frac{1}{n} \sum_{i=1}^n P_{S_i} D \bar \alpha
= \frac{1}{n} \sum_{i=1}^n \sum_{v \in S_i} \frac{[\bar \alpha]_v e_v}{p_v}
= \sum_{v=1}^{d} \left( \sum_{i\, | \, v \in S_i} 1 \right) \frac{[\bar \alpha]_v e_v}{n p_v}  \, ,
\end{align*}
where $e_v$ is the vector whose only nonzero component is the $v$ component which is equal to $1$.

By definition, $\sum_{i|v \in S_i} 1 = n p_v,$ which gives us Equation~\eqref{eq:apxBias}.

We define $\bar \alpha_i := \alpha_i - D_i\bar \alpha$ (contrary to~\citet{qsaga} where the authors define $\bar \alpha_i := \alpha_i - \bar \alpha$ since they do not concern themselves with sparsity).
Using the inequality $\|a+b\|^2 \leq 2 \|a\|^2 + 2 \|b\|^2$, we get:
\begin{align}
\Econd \|f'_i(x) - \bar \alpha_i\|^2 
\leq 2\Econd \|f'_i(x) -f'_i(x^*)\|^2 + 2\Econd \|\bar\alpha_i -f'_i(x^*)\|^2,
\end{align}
which is our equivalent to~\citet[Eq.(6)]{qsaga}, where only our definition of $\bar \alpha_i$ differs.

\paragraph{Deriving~\citet[Equation (9)]{qsaga}.}
We want to prove~\citet[Eq.(9)]{qsaga}:
\begin{align}
\Econd \|\bar \alpha_i -f'_i(x^*)\|^2
\leq \Econd \|\alpha_i -f'_i(x^*)\|^2 .
\end{align}

We have:
\begin{align} \label{eq:alphaiBarVariance}
\Econd \|\bar \alpha_i -f'_i(x^*)\|^2
&= \Econd \|\alpha_i -f'_i(x^*)\|^2
	-2\Econd \langle \alpha_i - f'_i(x^*), D_i \bar \alpha \rangle + \Econd\|D_i\bar \alpha\|^2 .
\end{align}

Let $D_{\neg i} := P_{S_i^c} D$; we then have the orthogonal decomposition $D \alpha = D_i \alpha + D_{\neg i} \alpha$ with $D_i \alpha \perp D_{\neg i} \alpha$, as they have disjoint support. We now use the orthogonality of $D_{\neg i} \alpha$ with any vector with support in $S_i$ to simplify the expression~\eqref{eq:alphaiBarVariance} as follows:
\begin{align}
\Econd \langle \alpha_i - f'_i(x^*), D_i\bar \alpha \rangle
&= \Econd \langle \alpha_i - f'_i(x^*), D_i \bar \alpha + D_{\neg i} \bar \alpha \rangle 
	\tag*{$(\alpha_i - f'_i(x^*) \perp D_{\neg i} \alpha)$}
\nonumber \\
&= \Econd \langle \alpha_i - f'_i(x^*), D \bar \alpha \rangle
\nonumber \\
&= \langle \Econd \big(\alpha_i - f'_i(x^*)\big), D \bar \alpha \rangle
\nonumber \\
&= \langle \Econd \alpha_i, D \bar \alpha \rangle
	\tag*{($f'(x^*) = 0$)}
\nonumber \\
&=\bar \alpha^\intercal D \bar \alpha \,  .
\end{align}

Similarly, 
\begin{align}
\Econd \|D_i\bar \alpha\|^2 
&= \Econd\langle D_i\bar \alpha, D_i\bar \alpha \rangle
\nonumber \\
&= \Econd\langle D_i\bar \alpha, D \bar \alpha \rangle
	\tag*{($D_i \alpha \perp D_{\neg i} \alpha$)}
\nonumber \\
&= \langle \Econd D_i\bar \alpha, D \bar \alpha \rangle
\nonumber \\
&= \bar \alpha^\intercal D \bar \alpha  \, .
\end{align}

Putting it all together,
\begin{align}
\Econd \|\bar \alpha_i -f'_i(x^*)\|^2 
= \Econd \|\alpha_i -f'_i(x^*)\|^2 - \bar \alpha^\intercal D \bar \alpha
\leq \Econd \|\alpha_i -f'_i(x^*)\|^2 .
\end{align}

This is our version of~\citet[Equation (9)]{qsaga}, which finishes the proof of~\citet[Lemma~1]{qsaga}. 
The rest of the proof from~\citet{qsaga} can then be reused without modification to obtain Theorem~\ref{th1}.
\qed

\section{Proof of Theorem~\ref{thm:convergence} and Corollary~\ref{thm:bigdata}}\label{apxB}
\subsection{Detailed outline}\label{apxB:outline}

We first give a detailed outline of the proof. The complete proof is given in the rest of Appendix~\ref{apxB}.

\paragraph{Initial recursive inequality.}
Let $g_t := g(\hat x_t, \hat \alpha^t, i_t)$. From the update equation~\eqref{eq:PIupdate} defining the virtual iterate $x_{t+1}$, the perturbed iterate framework~\citep{mania} gives:
\begin{align}\label{eq:initrec}
\|x_{t+1} - x^*\|^2 
&= \|x_t -\stepsize g_t -x^*\|^2 
\nonumber\\
&= \|x_t -x^*\|^2 + \stepsize^2 \|g_t\|^2  -2\stepsize\langle x_t -x^*,  g_t\rangle
\nonumber\\
&= \|x_t -x^*\|^2 + \stepsize^2 \|g_t\|^2 
	-2\stepsize\langle \hat x_t -x^*,  g_t\rangle +2\stepsize\langle \hat x_t -x_t,  g_t\rangle  \, .
\end{align}

Note that we have introduced $\hat x_t$ in the inner product because $g_t$ is a function of $\hat x_t$, not $x_t$.

In the sequential setting, we require $i_t$ to be independent of $x_t$ to get unbiasedness. 
In the perturbed iterate framework, we instead require that $i_t$ is independent of $\hat x_t$ (see Property~\ref{independence}). 
This crucial property enables us to use the unbiasedness condition~\eqref{eq:unbiasedness} to write:
$\E \langle \hat x_t -x^*,  g_t\rangle 
= \E \langle \hat x_t -x^*,  f'(\hat x_t)\rangle$. We thus take the expectation of~\eqref{eq:initrec} that allows us to use the $\mu$-strong convexity of $f$:\footnote{Note that here is our departure point with \citet{mania} who replaced the $f(\hat{x}_t)-f(x^*)$ term with the lower bound $\frac{\strongconvex}{2}\|\hat x_t - x^*\|^2$ in this relationship (see their Equation (2.4)), thus yielding an inequality too loose afterwards to get the fast rates for \SVRG.} 
\begin{align} \label{eq:strongconvexity}
\langle \hat x_t -x^*,  f'(\hat x_t)\rangle &\geq f(\hat x_t) -f(x^*) +\frac{\strongconvex}{2}\|\hat x_t - x^*\|^2 . 
\end{align}
With further manipulations on the expectation of~\eqref{eq:initrec}, including the use of the standard inequality $\|a + b\|^2 \leq 2\|a\|^2 + 2\|b\|^2$ (see Section~\ref{app:RecursiveDerivation}), we obtain our basic recursive contraction inequality:
\begin{align} \label{eq:RecursiveIneq2}
a_{t+1} &\leq 
	(1 -\frac{\stepsize \strongconvex}{2}) a_t 
	+ \stepsize^2 \E \|g_t\|^2 
	\underbrace{
		+ \stepsize\strongconvex \E\|\hat x_t - x^*\|^2 
		+ 2\stepsize \E \langle \hat x_t -x_t,  g_t\rangle
	}_{\text{additional asynchrony terms}}
	-2\stepsize e_t  \, ,
\end{align}
where $a_t := \E \|x_t - x^*\|^2$ and $e_t := \E f(\hat x_t) - f(x^*)$.

Inequality~\eqref{eq:RecursiveIneq2} is a midway point between the one derived in the proof of Lemma~1 in~\citet{qsaga} and Equation~(2.5) in~\citet{mania}, because we use the tighter strong convexity bound~\eqref{eq:strongconvexity} than in the latter (giving us the important extra term $-2\stepsize e_t$). 

In the sequential setting, one crucially uses the negative suboptimality term $-2\stepsize e_t$ to cancel the variance term $\stepsize^2 \E \|g_t\|^2$ (thus deriving a condition on $\stepsize$).
In our setting, we need to bound the additional asynchrony terms using the same negative suboptimality in order to prove convergence and speedup for our parallel algorithm -- this will give stronger constraints on the maximum step size.

The rest of the proof then proceeds as follows:
\begin{enumerate}
\item By using the expansion~\eqref{eq:async} for $\hat{x}_t-x_t$, we can bound the additional asynchrony terms in~\eqref{eq:RecursiveIneq2} in terms of the past updates ($\E \|g_u\|^2, u\leq t$). This gives Lemma~\ref{lma:1} below.
\item We then bound the updates $\E \|g_t\|^2$ in terms of past suboptimalities $(e_u)_{u \leq v}$ by using standard \SAGA\ inequalities and carefully analyzing the update rule for $\alpha_i^+$~\eqref{eq:SAGAupdate} in expectation. This gives Lemma~\ref{lma:suboptgt} below.
\item By substituting Lemma~\ref{lma:suboptgt} into Lemma~\ref{lma:1}, we get a master contraction inequality~\eqref{master} in terms of $a_{t+1}$, $a_t$ and $e_u, u\leq t$.
\item We define a novel Lyapunov function $\lyapunov_t = \sum_{u=0}^t (1-\contraction)^{t-u}a_u$ and manipulate the master inequality to show that $\lyapunov_t$ is bounded by a contraction, subject to a maximum step size condition on $\stepsize$ (given in Lemma~\ref{lma:3} below).
\item Finally, we unroll the Lyapunov inequality to get the convergence Theorem~\ref{thm:convergence}.
\end{enumerate}
We list the key lemmas below with their proof sketch, and give the detailed proof in the later sections of Appendix~\ref{apxB}.

\begin{lemma}[Inequality in terms of $g_t := g(\hat x_{t}, \hat \alpha^t, i_{t})$]\label{lma:1}
For all $t \geq 0$:
\begin{equation} \label{eq:recursivegt}
a_{t+1} \leq 
	(1 - \frac{\stepsize\strongconvex}{2}) a_t + \stepsize^2 C_1\E\|g_t\|^2 
	+ \stepsize^2 C_2\sum_{u=(t-\overlap)_+}^{t-1}\E\|g_{u}\|^2 - 2\stepsize e_t  \, ,
\end{equation}
where 
$C_1 := 1 + \sqrt{\sparsity}\overlap$ and
$C_2 :=  \sqrt{\sparsity} + \stepsize\strongconvex C_1$.
\end{lemma}

To prove this lemma we need to bound both $\E\|\hat x_t - x^*\|^2$ and $\E\langle \hat x_t -x_t,  g_t\rangle$ with respect to $(g_u, u\leq t)$. 
We achieve this by crucially using Equation~\eqref{eq:async}, together with the following proposition, which we derive by a combination of Cauchy-Schwartz and our sparsity definition (see Section~\ref{apxB:lma1}).
\begin{equation}
\E \langle G_{u}^t g_{u}, g_t \rangle \leq \frac{\sqrt{\sparsity}}{2}(\E\|g_{u}\|^2 + \E\|g_{t}\|^2)  \, .
\end{equation}

\begin{lemma} [Suboptimality bound on $\E \|g_t\|^2$]\label{lma:suboptgt}
For all $t \geq 0$,
\begin{equation}\label{gtbound}
\E\|g_t\|^2 
\leq 4\lipschitz e_t 
	+ \frac{4\lipschitz}{n} \sum_{u=1}^{t-1} (1 - \frac{1}{n})^{(t-2\overlap-u -1)_+} e_u
	+ 4\lipschitz (1 - \frac{1}{n})^{(t-\overlap)_+} \tilde e_0  \, .
\end{equation}
where $\tilde e_0 := \frac{1}{2\lipschitz} \E\|\alpha_i^0 - f'_i(x^*)\|^2$.\footnote{We introduce this quantity instead of $e_0$ so as to be able to handle the arbitrary initialization of the $\alpha_i^0$.}
\end{lemma}

From our Sparse \SAGA\ proof we know that (see Appendix~\ref{apxA}):
\begin{align}
\E\|g_t\|^2 
&\leq 2 \E \|f'_{i_t}(\hat x_t)-f'_{i_t}(x^*)\|^2
	+ 2 \E \|\hat \alpha_{i_t}^t - f'_{i_t}(x^*)\|^2 .
\end{align}
We can handle the first term by taking the expectation over a Lipschitz inequality (\citet[Equations (7) and (8)]{qsaga}. 
All that remains to prove the lemma is to express the $\E \|\hat \alpha_{i_t}^t - f'_{i_t}(x^*)\|^2$ term in terms of past suboptimalities.
We note that it can be seen as an expectation of past first terms with an adequate probability distribution which we derive and bound.

From our algorithm, we know that each dimension of the memory vector $[\hat \alpha_i]_v$ contains a partial gradient computed at some point in the past $[f'_i(\hat x_{u_{i, v}^{t}})]_v$\footnote{More precisely: $\forall t, i, v \hspace{0.5em}\exists u_{i, v}^t<t$ s.t. $[\hat \alpha_{i}^t]_v = [f'_{i}(\hat x_{u_{i, v}^{t}})]_v$.} (unless $u=0$, in which case we replace the partial gradient with $\alpha_i^0$).
We then derive bounds on $P(u_{i,v}^t = u)$ and sum on all possible $u$. 
Together with clever conditioning, we obtain Lemma~\ref{lma:suboptgt} (see Section~\ref{apxB:lma2}).

\paragraph{Master inequality.}
Let $H_t$ be defined as ${H_t: = \sum_{u=1}^{t-1} (1 - \frac{1}{n})^{(t-2\overlap-u -1)_+} e_u}$. 
Then, by setting~\eqref{gtbound} into Lemma~\ref{lma:1}, we get (see Section~\ref{apxB:master}):
\begin{equation}\label{master}
\begin{aligned}
a_{t+1} 
\leq &(1 - \frac{\stepsize\strongconvex}{2}) a_t 
	- 2\stepsize e_t 
	+ 4\lipschitz\stepsize^2 C_1 \big(e_t  + (1 - \frac{1}{n})^{(t-\overlap)_+} \tilde e_0 \big)
	+ \frac{4\lipschitz\stepsize^2 C_1}{n} H_t 
\\
	&+4\lipschitz\stepsize^2 C_2\sum_{u=(t-\overlap)_+}^{t-1} (e_u +  (1 - \frac{1}{n})^{(u - \overlap)_+} \tilde e_0 \big)
	+\frac{4\lipschitz\stepsize^2 C_2}{n} \sum_{u=(t-\overlap)_+}^{t-1} H_u  \, .
\end{aligned}
\end{equation}

\paragraph{Lyapunov function and associated recursive inequality.}
We now have the beginning of a contraction with additional positive terms which all converge to $0$ as we near the optimum, as well as our classical negative suboptimality term.
This is not unusual in the variance reduction literature. One successful approach in the sequential case is then to define a Lyapunov function which encompasses all terms and is a true contraction (see~\citet{SAGA, qsaga}).
We emulate this solution here.
However, while all terms in the sequential case only depend on the current iterate, $t$, in the parallel case we have terms ``from the past'' in our inequality.
To resolve this issue, we define a more involved Lyapunov function which also encompasses past iterates: 
\begin{align} \label{eq:LyapunovDefinition}
\lyapunov_t = \sum_{u=0}^t (1-\contraction)^{t-u}a_u, \quad 0<\contraction < 1,
\end{align}
where $\contraction$ is a target contraction rate that we define later.

Using the master inequality~\eqref{master}, we get (see Appendix~\ref{apxB:lyapunov}):
\begin{align}\label{Lyapunov}
\lyapunov_{t+1} 
&= (1 - \contraction)^{t+1}a_0 
	+ \sum_{u=0}^t(1 - \contraction)^{t-u}a_{u+1}
\nonumber \\
&\leq (1 - \contraction)^{t+1}a_0 + (1-\frac{\stepsize\strongconvex}{2})\lyapunov_t + \sum_{u=1}^t r_u^t e_u + r_0^t \tilde e_0  \, .
\end{align}

The aim is to prove that $\lyapunov_t$ is bounded by a contraction. 
We have two promising terms at the beginning of the inequality, and then we need to handle the last term.
Basically, we can rearrange the sums in~\eqref{master} to expose a simple sum of $e_u$ multiplied by factors $r_u^t$. 

Under specific conditions on $\contraction$ and $\stepsize$, we can prove that $r_u^t$ is negative for all $u \geq 1$, which coupled with the fact that each $e_u$ is positive means that we can safely drop the sum term from the inequality. 
The $r_0^t$ term is a bit trickier and is handled separately.

In order to have a bound on $e_t$ directly rather than on $\E \|\hat x_t - x^*\|^2$, we then introduce an additional $\stepsize e_t$ term on both sides of~\eqref{Lyapunov}.
The bound on $\stepsize$ under which the modified $r_t^t + \stepsize$ is negative is then twice as small (we could have used any multiplier between $0$ and $2\stepsize$, but chose $\stepsize$ for simplicity's sake).
This condition is given in the following Lemma.

\begin{lemma} [Sufficient condition for convergence]\label{lma:3}
Suppose $\overlap < n/10$ and $\contraction \leq 1/4n$. If 
\begin{align}\label{eq:lmacondition}
\stepsize \leq \stepsize^* = \frac{1}{32\lipschitz (1 + \sqrt{\sparsity} \overlap) \sqrt{1 + \frac{1}{8\kappa} \min(\overlap, \frac{1}{\sqrt{\sparsity}})}}
\end{align}
then for all $u \geq 1$, the $r_u^t$ from~\eqref{Lyapunov} verify:
\begin{align}
r_u^t \leq 0 \,; \quad r_t^t + \stepsize \leq 0 \,,
\end{align}
and thus we have:
\begin{align}
\stepsize  e_t + \lyapunov_{t+1} \leq (1 - \contraction)^{t+1}a_0 + (1-\frac{\stepsize\strongconvex}{2})\lyapunov_t + r_0^t \tilde e_0 \,.
\end{align}

\end{lemma}

We obtain this result after carefully deriving the $r_u^t$ terms. 
We find a second-order polynomial inequality in $\stepsize$, which we simplify down to~\eqref{eq:lmacondition} (see Appendix~\ref{apxB:lma3}).

We can then finish the argument to bound the suboptimality error $e_t$. We have:
\begin{align}
\lyapunov_{t+1} \leq
\stepsize e_t + \lyapunov_{t+1} 
&\leq (1-\frac{\stepsize\strongconvex}{2})\lyapunov_t  +(1 - \contraction)^{t+1} (a_0 + A \tilde e_0)  \, .
\end{align}

We have two linearly contracting terms. 
The sum contracts linearly with the worst rate between the two (the smallest geometric rate factor).
If we define $\contraction^* := \nu \min(\contraction, \stepsize \strongconvex / 2)$, with $0 <\nu < 1$,\footnote{$\nu$ is introduced to circumvent the problematic case where $\contraction$ and  $\stepsize \strongconvex / 2$ are too close together.} then we get:
\begin{align}
\stepsize e_t + \lyapunov_{t+1} 
&\leq (1-\frac{\stepsize\strongconvex}{2})^{t+1}\lyapunov_0 + (1 - \contraction^*)^{t+1} \frac{a_0 + A \tilde e_0}{1 -\eta} 
\\
\stepsize e_t &\leq (1 - \contraction^*)^{t+1} \big(\lyapunov_0 + \frac{1}{1 -\eta} (a_0 + A  \tilde e_0) \big) \, ,
\end{align}
where $\eta := \frac{1-M}{1-\contraction^*}$ with $M :=\max(\contraction, \stepsize\strongconvex/2)$. Our geometric rate factor is thus $\contraction^*$ (see Appendix~\ref{apxB:th2}).

\subsection{Extension to \SVRG}\label{apx:SVRGext}
Our proof can easily be adapted to accommodate the \SVRG\ variant introduced in~\citet{qsaga}, which is closer to \SAGA\ than the initial \SVRG\ algorithm and which is adaptive to local strong convexity (it does not require the inner loop epoch size~$m = \Omega(\kappa)$ as a hyperparameter). 
In this variant, instead of computing a full gradient every $m$ iterations, a random binary variable $U$ with probability $P(U=1)=1/n$ is sampled at the beginning of every iteration to determine whether a full gradient is computed or a normal \SVRG\ step is made. 
If $U=1$, then a full gradient is computed.
Otherwise the algorithm takes a normal inner \SVRG\ step.\footnote{Note that the parallel implementation is not very straightforward, as it requires a way to communicate to cores when they should start computing a batch gradient instead of inner steps.}

To prove convergence, all one has to do is to modify Lemma~\ref{lma:suboptgt} very slightly (the only difference is that the $(t -2\overlap -u -1)_+$ exponent is replaced by $(t - u)$ and the rest of the proof can be used as is).
The justification for this small tweak is that the batch steps in \SVRG\ are fully synchronized. More details can be found in Section~\ref{apxB:lma2} (see footnote~\ref{footnote}).

By using our ``after read'' labeling, we were also able to derive a convergence and speedup proof for the original \SVRG\ algorithm, but the proof technique diverges after Lemma~\ref{lma:1}. 
This is beyond the scope of this paper, so we omit it here. 
Using the ``after read'' labeling and a different proof technique from~\citet{mania}), we obtain an epoch size in $\mathcal{O}(\kappa)$ instead of $\mathcal{O}(\kappa^2)$ and a dependency in our overlap bound in $\mathcal{O}(\sparsity^{-1/2})$ instead of $\mathcal{O}(\sparsity^{-1/3})$.

\paragraph{\ASAGA\ vs. asynchronous \SVRG.}
There are several scenarios in which \ASAGA\ can be practically advantageous over its closely related cousin, asynchronous \SVRG\ (note though that ``asynchronous'' \SVRG\ still requires a synchronization step to compute the full gradients).

First, while \SAGA\ trades memory for less computation, in the case of generalized linear models the memory cost can be reduced to $\mathcal{O}(n)$, which is the same as for $\SVRG$.
This is of course also true for their asynchronous counterparts.

Second, as \ASAGA\ does not require any synchronization steps, it is better suited to heterogeneous computing environments (where cores have different clock speeds or are shared with other applications).

Finally, \ASAGA\ does not require knowing the condition number $\kappa$ for optimal convergence in the sparse regime.
It is thus adaptive to local strong convexity, whereas \SVRG\ is not.
Indeed, \SVRG\ and its asynchronous variant require setting an additional hyper-parameter -- the epoch size $m$ -- which needs to be at least $\Omega(\kappa)$ for convergence but yields a slower effective convergence rate than \ASAGA\ if it is set much bigger than $\kappa$.
\SVRG\ thus requires tuning this additional hyper-parameter or running the risk of either slower convergence (if the epoch size chosen is much bigger than the condition number) or even not converging at all (if $m$ is chosen to be much smaller than $\kappa$).\footnote{Note that as \SAGA\ (and contrary to the original \SVRG),  the \SVRG\ variant from~\citet{qsaga} does not require knowledge of $\kappa$ and is thus adaptive to local strong convexity, which carries over to its asynchronous adaptation.}

\subsection{Initial recursive inequality derivation} \label{app:RecursiveDerivation}
\vspace{-1mm}
We start by proving Equation~\eqref{eq:RecursiveIneq2}.
Let $g_t := g(\hat x_t, \hat \alpha^t, i_t)$. From~\eqref{eq:PIupdate}, we get:
\begin{align*}
\|x_{t+1} - x^*\|^2 \nonumber
= \|x_t -\stepsize g_t -x^*\|^2 
&= \|x_t -x^*\|^2 + \stepsize^2 \|g_t\|^2 -2\stepsize\langle x_t -x^*,  g_t\rangle
\\
&= \|x_t -x^*\|^2 + \stepsize^2 \|g_t\|^2 
	- 2 \stepsize\langle \hat x_t -x^*,  g_t\rangle +2\stepsize\langle \hat x_t -x_t,  g_t\rangle .
\end{align*}

In order to prove Equation~\eqref{eq:RecursiveIneq2}, we need to bound the $- 2 \stepsize\langle \hat x_t -x^*,  g_t\rangle$ term.
Thanks to Property~\ref{independence}, we can write:
\vspace{-2mm}
\begin{align*}
\E \langle \hat x_t -x^*,  g_t\rangle 
= \E \langle \hat x_t -x^*, \Econd g_t \rangle
= \E \langle \hat x_t -x^*,  f'(\hat x_t)\rangle  \, .
\end{align*}

We can now use a classical strong convexity bound as well as a squared triangle inequality to get:
\begin{align}
- \langle \hat x_t -x^*,  f'(\hat x_t)\rangle &\leq - \big(f(\hat x_t) -f(x^*)\big) -\frac{\strongconvex}{2}\|\hat x_t - x^*\|^2
\tag*{(Strong convexity bound)} \nonumber \\
- \|\hat x_t - x^*\|^2 &\leq \|\hat x_t - x_t\|^2 - \frac{1}{2} \|x_t - x^*\|^2
\tag*{($\|a+b\|^2 \leq 2 \|a\|^2 + 2 \|b\|^2$)} \nonumber \\
- 2 \stepsize \E \langle \hat x_t -x^*,  g_t\rangle &\leq
	- \frac{\stepsize \strongconvex}{2} \E \|x_t - x^*\|^2
	+ \stepsize \strongconvex \E \|\hat x_t - x_t\|^2
	-2 \stepsize \big(\E f(\hat x_t) - f(x^*)\big)  \, .
\end{align}

Putting it all together, we get the initial recursive inequality~\eqref{eq:RecursiveIneq2}, rewritten here explicitly:
\begin{align}
a_{t+1} \leq 
	(1 -\frac{\stepsize \strongconvex}{2}) a_t 
	+ \stepsize^2 \E \|g_t\|^2 
	+ \stepsize\strongconvex \E\|\hat x_t - x_t\|^2 
	+ 2\stepsize \E \langle \hat x_t -x_t,  g_t\rangle
	-2\stepsize e_t  \, ,
\end{align}
where $a_t := \E \|x_t - x^*\|^2$ and $e_t := \E f(\hat x_t) - f(x^*)$.

\subsection{Proof of Lemma~\ref{lma:1}} \label{apxB:lma1}
To prove Lemma~\ref{lma:1}, we now bound both $\E\|\hat x_t - x_t\|^2$ and $\E\langle \hat x_t -x_t,  g_t\rangle$ with respect to $\E\|g_u\|^2, u\leq t$.

We start by proving a relevant property of $\sparsity$, which enables us to derive an essential inequality for both these terms, given in Proposition~\ref{prop:1} below.
We reuse the sparsity constant introduced in~\citet{smola} and relate it to the one we have defined earlier, $\sparsityr$:
\begin{remark} \label{rmk:1}
Let $D$ be the smallest constant such that:
\begin{align} \label{sparsitycondition}
\Econd \|x\|_i^2 = \frac{1}{n} \sum_{i=1}^n \|x\|_i^2 \leq D \|x\|^2 \quad  \forall x \in \mathbb{R}^d,
\end{align}
where $\|.\|_i$ is defined to be the $\ell_2$-norm restricted to the support $S_i$ of $f_i$. 
We have:
\begin{equation}
D = \frac{\sparsityr}{n} = \sparsity  \, .
\end{equation}
\end{remark}

\begin{proof}
We have:
\begin{align}
\Econd \|x\|_i^2 = \frac{1}{n} \sum_{i=1}^n \|x\|_i^2
= \frac{1}{n} \sum_{i=1}^n \sum_{v \in S_i} [x]_v^2
= \frac{1}{n} \sum_{v=1}^d \sum_{i \mid v \in S_i} [x]_v^2
= \frac{1}{n} \sum_{v=1}^d \delta_v [x]_v^2 \, ,
\end{align}
where $\delta_v := \mathbf{card}(i \mid v \in S_i)$.

This implies:
\begin{align}
D \geq \frac{1}{n} \sum_{v=1}^d \delta_v \frac{[x]_v^2}{\|x\|^2}  \, .
\end{align}

Since $D$ is the minimum constant satisfying this inequality, we have:
\begin{align}
D = \max_{x \in \mathbb{R}^d} \frac{1}{n} \sum_{v=1}^d \delta_v \frac{[x]_v^2}{\|x\|^2}  \, .
\end{align}
We need to find $x$ such that it maximizes the right-hand side term.
Note that the vector $([x]_v^2 / \|x\|^2)_{v=1..d}$ is in the unit probability simplex, which means that an equivalent problem is the maximization over all convex combinations of $(\delta_v)_{v=1..d}$.
This maximum is found by putting all the weight on the maximum $\delta_v$, which is $\sparsityr$ by definition.

This means that $\sparsity = \sparsityr / n$ is indeed the smallest constant satisfying~\eqref{sparsitycondition}.
\end{proof}

\begin{proposition}\label{prop:1}
For any $u \neq t$,
\begin{align}\label{sparseproduct}
\E |\langle g_{u}, g_t \rangle | &\leq \frac{\sqrt{\sparsity}}{2}(\E\|g_{u}\|^2 + \E\|g_{t}\|^2)  \, .
\end{align}
\end{proposition}

\begin{proof}
Let $u \neq t$. Without loss of generality, $u < t$.\footnote{One only has to switch $u$ and $t$ if $u>t$.}
Then:

\begin{align}
\E |\langle g_{u}, g_t \rangle |
&\leq \E\|g_{u}\|_{i_t}\|g_t\| 
\tag*{(Sparse inner product; support of $g_t$ is $S_{i_t}$)} \nonumber \\ 
&\leq \sqrt{\E\|g_{u}\|_{i_t}^2}\sqrt{\E\|g_{t}\|^2}
\tag*{(Cauchy-Schwarz for expectations)} \nonumber \\ 
&\leq \sqrt{\sparsity \E\|g_{u}\|^2}\sqrt{\E\|g_{t}\|^2}
\tag*{(Remark~\ref{rmk:1} and $i_t \perp\!\!\!\perp g_u, \forall u < t$)} \nonumber \\ 
&\leq \frac{\sqrt{\sparsity}}{2}(\E\|g_{u}\|^2 + \E\|g_{t}\|^2)  \, .
\tag*{(AM-GM inequality)}
\end{align}
All told, we have:
\begin{align}
\E |\langle g_{u}, g_t \rangle | &\leq \frac{\sqrt{\sparsity}}{2}(\E\|g_{u}\|^2 + \E\|g_{t}\|^2)  \, .
\end{align}
\end{proof}

\paragraph{Bounding $\E\langle \hat x_t -x_t,  g_t\rangle$ in terms of $g_u$.}
\begin{align}
\frac{1}{\stepsize} \E\langle \hat x_t - x_t, g_t \rangle 
&= \sum_{u=(t - \overlap)_+}^{t-1} \E \langle G_u^t g_u, g_t \rangle
\tag*{(by Equation~\eqref{eq:async})} \nonumber \\
&\leq \sum_{u=(t - \overlap)_+}^{t-1} \E | \langle g_u, g_t \rangle |
\tag*{($G_u^t$ diagonal matrices with terms in $\{0, 1\}$)}\nonumber \\
&\leq \sum_{u=(t - \overlap)_+}^{t-1} \frac{\sqrt{\sparsity}}{2}(\E\|g_{u}\|^2 + \E\|g_{t}\|^2) 
\tag*{(by Proposition~\ref{prop:1})}\nonumber \\
&\leq \frac{\sqrt{\sparsity}}{2} \sum_{u=(t - \overlap)_+}^{t-1}\E\|g_{u}\|^2 + \frac{\sqrt{\sparsity}\overlap}{2}\E\|g_{t}\|^2 .
\end{align}

\paragraph{Bounding $\E\|\hat x_t - x_t\|^2$ with respect to $g_u$}
Thanks to the expansion for $\hat x_t - x_t$~\eqref{eq:async}, we get:
\begin{align*}
\|\hat x_t - x_t\|^2 
\leq \stepsize^2 \sum_{u, v=(t -\overlap)_+}^{t-1}|\langle G_u^t g_{u}, G_v^t g_{v}\rangle | 
\leq \stepsize^2 \sum_{u=(t -\overlap)_+}^{t-1}\|g_{u}\|^2 
	+ \stepsize^2 \sum_{\substack{u, v=(t-\overlap)_+ \\u\neq v}}^{t-1} |\langle G_u^t g_{u}, G_v^t g_{v}\rangle |  \, .
\end{align*}
Using~\eqref{sparseproduct} from Proposition~\ref{prop:1}, we have that for $u \neq v$:

\begin{equation} \label{supersparse}
\E |\langle G_u^t g_{u}, G_v^t g_{v}\rangle | 
\leq \E |\langle g_{u}, g_{v}\rangle | 
\leq \frac{\sqrt{\sparsity}}{2}(\E\|g_{u}\|^2 + \E\|g_{v}\|^2) \, .
\end{equation}

By taking the expectation and using~\eqref{supersparse}, we get:
\begin{align}
\E\|\hat x_t - x_t\|^2
&\leq \stepsize^2 \sum_{u=(t-\overlap)_+}^{t-1}\E\|g_{u}\|^2 
	+ \stepsize^2 \sqrt{\sparsity}(\overlap-1)_+ \sum_{u=(t-\overlap)_+}^{t-1}\E\|g_{u}\|^2 
\nonumber \\
&= \stepsize^2 \big(1+\sqrt{\sparsity}(\overlap-1)_+ \big)\sum_{u=(t-\overlap)_+}^{t-1}\E\|g_{u}\|^2 
\nonumber \\
&\leq \stepsize^2 \big(1+\sqrt{\sparsity}\overlap \big)\sum_{u=(t-\overlap)_+}^{t-1}\E\|g_{u}\|^2 .
\end{align}

We can now rewrite~\eqref{eq:RecursiveIneq2} in terms of $\E\|g_t\| ^2$, which finishes the proof for Lemma~\ref{lma:1} (by introducing $C_1$ and $C_2$ as specified in Lemma~\ref{lma:1}):
\begin{align}
a_{t+1} &\leq 
	(1 - \frac{\stepsize\strongconvex}{2}) a_t 
	- 2\stepsize e_t
	+ \stepsize^2 \E\|g_t\|^2
	+ \stepsize^3 \strongconvex(1+\sqrt{\sparsity}\overlap)\sum_{u=(t-\overlap)_+}^{t-1}\E\|g_{u}\|^2
\nonumber \\
	&\quad + \stepsize^2 \sqrt{\sparsity}\sum_{u=(t-\overlap)_+}^{t-1}\E\|g_{u}\|^2
	+ \stepsize^2 \sqrt{\sparsity}\overlap\E\|g_t\|^2
\nonumber \\
&\leq (1 - \frac{\stepsize\strongconvex}{2}) a_t 
	- 2\stepsize e_t
	+ \stepsize^2 C_1 \E\|g_t\|^2
	+ \stepsize^2 C_2 \sum_{u=(t-\overlap)_+}^{t-1}\E\|g_{u}\|^2 .
\end{align}
\qed

\subsection{Proof of Lemma~\ref{lma:suboptgt}} \label{apxB:lma2}
We now derive our bound on $g_t$ with respect to suboptimality.
From Appendix~\ref{apxA}, we know that:
\begin{align}
\E\|g_t\|^2
&\leq 2 \E \|f'_{i_t}(\hat x_t)-f'_{i_t}(x^*)\|^2 
	+ 2 \E \|\hat \alpha_{i_t}^t - f'_{i_t}(x^*)\|^2 \label{eq:classicsaga}
\\
\E \|f'_{i_t}(\hat x_t)-f'_{i_t}(x^*)\|^2 
&\leq 2\lipschitz\big(\E f(\hat x_t) - f(x^*)\big)
= 2\lipschitz e_t  \, . \label{eq:classicsaga2}
\end{align}

\textbf{N. B.: In the following, $i_t$ is a random variable picked uniformly at random in $\{1,...,n\}$, whereas $i$ is a fixed constant.} 

We still have to handle the $\E \|\hat \alpha_{i_t}^t - f'_{i_t}(x^*)\|^2$ term and express it in terms of past suboptimalities. 
We know from our definition of $t$ that $i_t$ and $\hat x_u$ are independent $\forall u<t$.
Given the ``after read'' global ordering, $\Econd$ -- the expectation on $i_t$ conditioned on $\hat x_t$ and all ``past" $\hat x_u$ and $i_u$ -- is well defined, and we can rewrite our quantity as:
\begin{align*}
\E \|\hat \alpha_{i_t}^t - f'_{i_t}(x^*)\|^2 
= \E \big( \Econd \|\hat \alpha_{i_t}^t - f'_{i_t}(x^*)\|^2 \big)
&= \E \frac{1}{n} \sum_{i=1}^n  \|\hat \alpha_i^t - f'_i(x^*)\|^2
\\
&=  \frac{1}{n} \sum_{i=1}^n \E \|\hat \alpha_i^t - f'_i(x^*)\|^2 .
\end{align*}

Now, with $i$ fixed, let $u_{i,l}^t$ be the time of the iterate last used to write the $[\hat \alpha_{i}^t]_l$ quantity, i.e. $[\hat \alpha_{i}^t]_l = [f'_{i}(\hat x_{u_{i, l}^t})]_l$.
We know\footnote{In the case where $u=0$, one would have to replace the partial gradient with $\alpha_i^0$. We omit this special case here for clarity of exposition.} that $0 \leq u_{i,l}^t \leq t - 1$.
To use this information, we first need to split $\hat \alpha_i$ along its dimensions to handle the possible inconsistencies among them:
\begin{align*}
\E \|\hat \alpha_i^t - f'_i(x^*)\|^2
=
\E \sum_{l=1}^d \big([\hat \alpha_i^t]_l-[f'_i(x^*)]_l\big)^2
= 
\sum_{l=1}^d \E \Big[ \big([\hat \alpha_i^t]_l-[f'_i(x^*)]_l\big)^2 \Big] .
\end{align*}

This gives us:
\begin{align}
\E \|\hat \alpha_i^t - f'_i(x^*)\|^2
&= 
\sum_{l=1}^d \E \Big[ \big(f'_{i}(\hat x_{u_{i, l}^t})_l-f'_i(x^*)_l\big)^2 \Big] \notag \\
&=
\sum_{l=1}^d\E \Big[\sum_{u=0}^{t-1} \ind_{\{u_{i, l}^t = u\}} \big(f'_i(\hat x_u)_l-f'_i(x^*)_l\big)^2 \Big] \notag \\
&=
\sum_{u=0}^{t-1} \sum_{l=1}^d\E \Big[\ind_{\{u_{i, l}^t = u\}} \big(f'_i(\hat x_u)_l-f'_i(x^*)_l\big)^2 \Big] 
\label{eq:IndicatorsAppearance}.
\end{align}

We will now rewrite the indicator so as to obtain independent events from the rest of the equality. 
This will enable us to distribute the expectation. 
Suppose $u>0$ ($u=0$ is a special case which we will handle afterwards). $\{u_{i, l}^t = u\}$ requires two things:
\begin{enumerate}
\item at time $u$, $i$ was picked uniformly at random,
\item (roughly) $i$ was not picked again between $u$ and $t$.
\end{enumerate}
We need to refine both conditions because we have to account for possible collisions due to asynchrony. 
We know from our definition of $\overlap$ that the $t^\mathrm{th}$ iteration finishes before at $t + \overlap + 1$, but it may still be unfinished by time $t + \overlap$.
This means that we can only be sure that an update selecting $i$ at time $v$ has been written to memory at time $t$ if $v \leq t -\overlap -1$.
Later updates may not have been written yet at time $t$.
Similarly, updates before $v = u + \overlap +1$ may be overwritten by the $u^\mathrm{th}$ update so we cannot infer that they did not select $i$. From this discussion, we conclude that $u_{i, l}^t = u$ implies that $i_v \neq i$ for all $v$ between $u+\overlap+1$ and $t-\overlap-1$, though it can still happen that $i_v = i$ for $v$ outside this range.

Using the fact that $i_u$ and $i_v$ are independent for $v \neq u$, we can thus upper bound the indicator function appearing in~\eqref{eq:IndicatorsAppearance} as follows:\footnote{\label{footnote}In the simpler case of the variant of \SVRG\ from~\citet{qsaga} as described in~\ref{apx:SVRGext}, the batch gradient computations are fully synchronized. This means that we can write much the same inequality without having to worry about possible overwrites, thus replacing  $\ind_{\{i_v \neq i\ \forall v\ \text{s.t.}\ u+\overlap+1 \leq v \leq t-\overlap-1\}}$ by $\ind_{\{i_v \neq i\ \forall v\ \text{s.t.}\ u+1 \leq v \leq t\}}$.}
\begin{align}\label{eq:indicatrices}
\ind_{\{u_{i,l}^t = u\}} 
\leq \ind_{\{i_u=i\}} 
	 \ind_{\{i_v \neq i\ \forall v\ \text{s.t.}\ u+\overlap+1 \leq v \leq t-\overlap-1\}} .
\end{align}
This gives us:
\begin{align}
\E \Big[ &\ind_{\{u_{i,l}^t = u\}} \big(f'_i(\hat x_u)_l-f'_i(x^*)_l\big)^2 \Big]
\nonumber \\ 
&\leq \E \Big[\ind_{\{i_u=i\}} 
	\ind_{\{i_v \neq i\ \forall v\ \text{s.t.}\ u+\overlap+1 \leq v \leq t-\overlap-1\}} \big(f'_i(\hat x_u)_l-f'_i(x^*)_l\big)^2\Big]
\nonumber \\ 
&\leq P\{i_u=i\}
	P\{i_v \neq i\ \forall v\ \text{s.t.}\ u+\overlap+1 \leq v \leq t-\overlap-1\}
	\E\big(f'_i(\hat x_u)_l-f'_i(x^*)_l\big)^2
\tag*{($i_v \perp\!\!\! \perp \hat x_u, \forall v \geq u$)} \nonumber \\ 
&\leq \frac{1}{n}(1 - \frac{1}{n})^{(t-2\overlap-u -1)_+}
	\E\big(f'_i(\hat x_u)_l-f'_i(x^*)_l\big)^2 .
\end{align}
Note that the third line used the crucial independence assumption $i_v \perp\!\!\! \perp \hat x_u, \forall v \geq u$ arising from our ``After Read'' ordering.
Summing over all dimensions $l$, we then get:
\begin{align}
\E \Big[ \ind_{\{u_{i,l}^t = u\}} \|f'_i(\hat x_u)-f'_i(x^*)\|^2 \Big]
\leq \frac{1}{n}(1 - \frac{1}{n})^{(t-2\overlap-u-1)_+}
	\E \|f'_i(\hat x_u)-f'_i(x^*)\|^2 .
\end{align}

So now:
\begin{align}
\E \|\hat \alpha_{i_t}^t - f'_{i_t}(x^*)\|^2 - \lambda \tilde e_0
&\leq \frac{1}{n}\sum_{i=1}^n \sum_{u=1}^{t-1} \frac{1}{n}(1 - \frac{1}{n})^{(t-2\overlap-u-1)_+} \E \|f'_i(\hat x_{u}) - f'_i(x^*)\|^2 
\nonumber \\ 
&= \sum_{u=1}^{t-1} \frac{1}{n}(1 - \frac{1}{n})^{(t-2\overlap-u -1)_+} \frac{1}{n}\sum_{i=1}^n \E \|f'_i(\hat x_{u}) - f'_i(x^*)\|^2
\nonumber \\ 
&= \sum_{u=1}^{t-1} \frac{1}{n}(1 - \frac{1}{n})^{(t-2\overlap-u -1)_+} \E \Big(\Econd \|f'_{i_u}(\hat x_{u}) - f'_{i_u}(x^*)\|^2 \Big)
\tag*{($i_u \perp\!\!\!\perp \hat x_u$)} \nonumber \\ 
&\leq \frac{2\lipschitz}{n} \sum_{u=1}^{t-1} (1 - \frac{1}{n})^{(t-2\overlap-u -1)_+} e_u
\tag*{(by Equation~\eqref{eq:classicsaga2})} \nonumber \\ 
&= \frac{2\lipschitz}{n} \sum_{u=1}^{(t-2\overlap -1)_+} (1 - \frac{1}{n})^{t-2\overlap-u -1} e_u 
	+ \frac{2\lipschitz}{n} \sum_{u=\max(1, t-2\overlap)}^{t-1} e_u  \, .\label{eq:51}
\end{align}

Note that we have excluded $\tilde e_0$ from our formula, using a generic $\lambda$ multiplier. 
We need to treat the case $u=0$ differently to bound $\ind_{\{u_{i,l}^t = u\}}$.
Because all our initial $\alpha_i$ are initialized to a fixed $\alpha_i^0$, $\{u_i^t = 0\}$ just means that $i$ has not been picked between $0$ and $t-\overlap -1$, i.e. $\{i_v \neq i\ \forall\ v\ \text{s.t.}\ 0 \leq v \leq t - \overlap -1\}$. 
This means that the $\ind_{\{i_u=i\}}$ term in~\eqref{eq:indicatrices} disappears and thus we lose a $\frac{1}{n}$ factor compared to the case where $u>1$.

Let us now evaluate $\lambda$.
We have:
\begin{align}\label{eq:52}
\E \Big[\ind_{\{u_i^t = 0\}} \|\alpha_i^0-f'_i(x^*)\|^2 \Big]
&\leq \E \Big[\ind_{\{i_v \neq i\ \forall\ v\ \text{s.t.}\ 0 \leq v \leq t - \overlap -1\}} \|\alpha_i^0-f'_i(x^*)\|^2 \Big]
\nonumber \\ 
&\leq P\{i_v \neq i\ \forall\ v\ \text{s.t.}\ 0 \leq v \leq t - \overlap -1\} \E \|\alpha_i^0-f'_i(x^*)\|^2
\nonumber \\ 
&\leq (1 - \frac{1}{n})^{(t-\overlap)_+} \E \|\alpha_i^0-f'_i(x^*)\|^2 .
\end{align}

Plugging~\eqref{eq:51} and~\eqref{eq:52} into~\eqref{eq:classicsaga}, we get Lemma~\ref{lma:suboptgt}:
\begin{align}\label{eq:53}
\E\|g_t\|^2
\leq 4\lipschitz e_t 
	+ \frac{4\lipschitz}{n} \sum_{u=1}^{t-1} (1 - \frac{1}{n})^{(t-2\overlap-u -1)_+} e_u 
	+ 4\lipschitz (1 - \frac{1}{n})^{(t-\overlap)_+} \tilde e_0 \,  ,
\end{align}
where we have introduced $\tilde e_0 := \frac{1}{2\lipschitz} \E\|\alpha_i^0 - f'_i(x^*)\|^2$.
Note that in the original \SAGA\ algorithm, a batch gradient is computed to set the $\alpha_i^0 = f'_i(x_0)$. 
In this setting, we can write Lemma~\ref{lma:suboptgt} using $\tilde e_0 \leq e_0$ thanks to~\eqref{eq:classicsaga2}.
In the more general setting where we initialize all $\alpha_i^0$ to a fixed quantity, we cannot use~\eqref{eq:classicsaga2} to bound $\E\|\alpha_i^0 - f'_i(x^*)\|^2$ which means that we have to introduce $\tilde e_0$.

\subsection{Master inequality derivation}\label{apxB:master}
Now, if we combine the bound on $\E\|g_{t}\|^2$ which we just derived (i.e. Lemma~\ref{lma:suboptgt}) with Lemma~\ref{lma:1}, we get:
\begin{equation}
\begin{aligned}
a_{t+1} 
\leq &(1 - \frac{\stepsize\strongconvex}{2}) a_t 
	- 2\stepsize e_t
\\
	&+ 4\lipschitz\stepsize^2 C_1e_t 
	+ \frac{4\lipschitz\stepsize^2 C_1}{n} \sum_{u=1}^{t-1} (1 - \frac{1}{n})^{(t-2\overlap-u -1)_+} e_u
	+ 4\lipschitz \stepsize^2 C_1(1 - \frac{1}{n})^{(t-\overlap)_+} \tilde e_0
\\
	&+4\lipschitz\stepsize^2 C_2\sum_{u=(t-\overlap)_+}^{t-1} e_u 
	+4\lipschitz\stepsize^2 C_2 \sum_{u=(t-\overlap)_+}^{t-1} (1 - \frac{1}{n})^{(u - \overlap)_+} \tilde e_0
\\
	&+ \frac{4\lipschitz\stepsize^2 C_2}{n} \sum_{u=(t-\overlap)_+}^{t-1} \sum_{v=1}^{u-1} (1-\frac{1}{n})^{(u - 2\overlap - v -1)_+}e_v  \, .
\end{aligned}
\end{equation}

If we define $H_t := \sum_{u=1}^{t-1} (1 - \frac{1}{n})^{(t-2\overlap-u-1)_+} e_u$, then we get:
\begin{equation}\label{eq:master}
\begin{aligned} 
a_{t+1} 
\leq &(1 - \frac{\stepsize\strongconvex}{2}) a_t 
	- 2\stepsize e_t
\\
	&+ 4\lipschitz\stepsize^2 C_1 \big(e_t  + (1 - \frac{1}{n})^{(t-\overlap)_+} \tilde e_0 \big)
	+ \frac{4\lipschitz\stepsize^2 C_1}{n} H_t
\\
	&+4\lipschitz\stepsize^2 C_2\sum_{u=(t-\overlap)_+}^{t-1} (e_u +  (1 - \frac{1}{n})^{(u - \overlap)_+} \tilde e_0 \big)
	+\frac{4\lipschitz\stepsize^2 C_2}{n} \sum_{u=(t-\overlap)_+}^{t-1} H_u  \, ,
\end{aligned}
\end{equation}
which is the master inequality~\eqref{master}.

\subsection{Lyapunov function and associated recursive inequality}\label{apxB:lyapunov}
We define $\lyapunov_t := \sum_{u=0}^t (1-\contraction)^{t-u}a_u$ for some target contraction rate $\contraction < 1$ to be defined later.
We have:
\begin{align}
\lyapunov_{t+1} 
&= (1 - \contraction)^{t+1}a_0 
	+ \sum_{u=1}^{t+1}(1 - \contraction)^{t+1-u}a_u
= (1 - \contraction)^{t+1}a_0 
	+ \sum_{u=0}^t(1 - \contraction)^{t-u}a_{u+1} \,  .
\end{align}

We now use our new bound on $a_{t+1}$,~\eqref{eq:master}:
\begin{align} \label{eq:rutMaster}
\lyapunov_{t+1} 
&\leq (1 - \contraction)^{t+1}a_0 
	+ \sum_{u = 0}^t(1 - \contraction)^{t-u} \Big[
		(1 - \frac{\stepsize\strongconvex}{2}) a_u 
		- 2\stepsize e_u
		+ 4\lipschitz\stepsize^2 C_1 \big(e_u  + (1 - \frac{1}{n})^{(u-\overlap)_+} \tilde e_0 \big)
\notag \\		
		&\qquad \qquad \qquad \qquad \qquad \qquad \qquad 
		+ \frac{4\lipschitz\stepsize^2 C_1}{n} H_u
		+\frac{4\lipschitz\stepsize^2 C_2}{n} \sum_{v=(u-\overlap)_+}^{u-1} H_v	
\notag \\		
		&\qquad \qquad \qquad \qquad \qquad \qquad \qquad 
		+4\lipschitz\stepsize^2 C_2\sum_{v=(u-\overlap)_+}^{u-1} (e_v +  (1 - \frac{1}{n})^{(v - \overlap)_+} \tilde e_0 \big)
	\Big]
\notag \\
&\leq (1 - \contraction)^{t+1}a_0 
	+ (1-\frac{\stepsize\strongconvex}{2})\lyapunov_t 
\notag \\
	&\qquad \qquad + \sum_{u = 0}^t(1 - \contraction)^{t-u} \Big[
		- 2\stepsize e_u
		+ 4\lipschitz\stepsize^2 C_1 \big(e_u  + (1 - \frac{1}{n})^{(u-\overlap)_+} \tilde e_0 \big)
\notag \\		
		&\qquad \qquad \qquad \qquad \qquad \qquad
		+ \frac{4\lipschitz\stepsize^2 C_1}{n} H_u
		+\frac{4\lipschitz\stepsize^2 C_2}{n} \sum_{v=(u-\overlap)_+}^{u-1} H_v	
\notag \\
		&\qquad \qquad \qquad \qquad \qquad \qquad
		+4\lipschitz\stepsize^2 C_2\sum_{v=(u-\overlap)_+}^{u-1} (e_v +  (1 - \frac{1}{n})^{(v - \overlap)_+} \tilde e_0 \big)
	\Big] .
\end{align}

We can now rearrange the sums to expose a simple sum of $e_u$ multiplied by factors $r_u^t$:
\begin{align}\label{apx:Lyapunov}
\lyapunov_{t+1} \leq (1 - \contraction)^{t+1}a_0 + (1-\frac{\stepsize\strongconvex}{2})\lyapunov_t + \sum_{u=1}^t r_u^t e_u + r_0^t \tilde e_0  \, .
\end{align}

\subsection{Proof of Lemma~\ref{lma:3}}\label{apxB:lma3}
We want to make explicit what conditions on $\contraction$ and $\stepsize$ are necessary to ensure that $r_u^t$ is negative for all $u \geq 1$. 
Since each $e_u$ is positive, we will then be able to safely drop the sum term from the inequality. 
The $r_0^t$ term is a bit trickier and is handled separately. 
Indeed, trying to enforce that $r_0^t$ is negative results in a significantly worse condition on $\stepsize$ and eventually a convergence rate smaller by a factor of $n$ than our final result.
Instead, we handle this term directly in the Lyapunov function.

\paragraph{Computation of $r_u^t$.}
Let's now make the multiplying factor explicit. 
We assume $u \geq 1$.

We split $r_u^t$ into five parts coming from~\eqref{eq:rutMaster}: 
\begin{itemize}
\item $r_1$, the part coming from the $-2\stepsize e_u$ terms;
\item $r_2$, coming from $4\lipschitz\stepsize^2 C_1 e_u$;
\item $r_3$, coming from $\frac{4\lipschitz\stepsize^2 C_1}{n} H_u$;
\item $r_4$, coming from $4\lipschitz\stepsize^2 C_2\sum_{v=(u-\overlap)_+}^{u-1} e_v$;
\item $r_5$, coming from $\frac{4\lipschitz\stepsize^2 C_2}{n} \sum_{v=(u-\overlap)_+}^{u-1} H_v$.
\end{itemize}

$r_1$ is easy to derive. Each of these terms appears only in one inequality. 
So for $u$ at time $t$, the term is:
\begin{equation}\label{eq:r1}
r_1 = -2 \stepsize (1- \contraction)^{t-u} .
\end{equation}

For much the same reasons, $r_2$ is also easy to derive and is: 
\begin{equation}\label{eq:r2}
r_2 = 4\lipschitz\stepsize^2 C_1(1- \contraction)^{t-u} .
\end{equation}

$r_3$ is a bit trickier, because for a given $v > 0$ there are several $H_u$ which contain $e_v$. 
The key insight is that we can rewrite our double sum in the following manner:
\begin{align}
\sum_{u=0}^t(1 - \contraction)^{t-u} &\sum_{v=1}^{u-1} (1-\frac{1}{n})^{(u - 2\overlap -v -1)_+} e_v 
\nonumber \\
&= \sum_{v=1}^{t-1} e_v \sum_{u=v+1}^{t}(1 - \contraction)^{t-u}  (1-\frac{1}{n})^{(u - 2\overlap -v -1)_+}
\nonumber \\
&\leq \sum_{v=1}^{t-1} e_v \Big[
		\sum_{u=v+1}^{\min(t, v + 2\overlap)}(1 - \contraction)^{t-u} 
		+ \sum_{u=v + 2 \overlap +1}^{t}(1 - \contraction)^{t-u}  (1-\frac{1}{n})^{u - 2\overlap -v -1}
	\Big]
\nonumber \\
&\leq \sum_{v=1}^{t-1} e_v \Big[
		2\overlap (1 - \contraction)^{t-v-2\overlap}
		+ (1 - \contraction)^{t-v-2\overlap-1} \sum_{u=v + 2 \overlap +1}^{t}q^{u - 2\overlap -v -1}
	\Big]
\nonumber \\
&\leq \sum_{v=1}^{t-1} (1 - \contraction)^{t-v} e_v (1 - \contraction)^{-2\overlap -1} \big[
		2\overlap
		+ \frac{1}{1-q}
	\big],
\end{align}
where we have defined: 
\begin{equation} \label{eq:qDefinition}
q :=\frac{1-1/n}{1-\contraction}, \quad \text{with the assumption $\contraction < \frac{1}{n}$}  \, .
\end{equation}
Note that we have bounded the $\min(t, v+2\overlap)$ term by $v+2\overlap$ in the first sub-sum, effectively adding more positive terms. 

This gives us that at time $t$, for $u$: 
\begin{equation}\label{eq:r3}
r_3 \leq \frac{4\lipschitz\stepsize^2 C_1}{n}(1 - \contraction)^{t-u} (1 - \contraction)^{-2\overlap -1} \big[2\overlap + \frac{1}{1-q}\big] .
\end{equation}

For $r_4$ we use the same trick:
\begin{align}
\sum_{u=0}^t (1-\contraction)^{t-u} \sum_{v=(u-\overlap)_+}^{u-1} e_v
&= \sum_{v=0}^{t-1} e_v  \sum_{u=v+1}^{\min(t, v+\overlap)} (1-\contraction)^{t-u}
\nonumber \\
&\leq \sum_{v=0}^{t-1} e_v  \sum_{u=v+1}^{v+\overlap} (1-\contraction)^{t-u}
\leq \sum_{v=0}^{t-1} e_v \overlap (1-\contraction)^{t-v-\overlap} .
\end{align}

This gives us that at time $t$, for $u$: 
\begin{equation}\label{eq:r4}
r_4 \leq 4\lipschitz\stepsize^2 C_2(1 - \contraction)^{t-u} \overlap (1 - \contraction)^{-\overlap}\,  .
\end{equation}

Finally we compute $r_5$ which is the most complicated term. 
Indeed, to find the factor of $e_w$ for a given $w > 0$, one has to compute a triple sum, $\sum_{u = 0}^t(1 - \contraction)^{t-u} \sum_{v=(u-\overlap)_+}^{u-1} H_v$.
We start by computing the factor of $e_w$ in the inner double sum, $\sum_{v=(u-\overlap)_+}^{u-1} H_v$.

\begin{align}
\sum_{v=(u-\overlap)_+}^{u-1} \sum_{w=1}^{v-1}(1-\frac{1}{n})^{(v -2\overlap -w -1)_+} e_w
= \sum_{w=1}^{u-2} e_w \sum_{v=\max(w+1, u-\overlap)}^{u-1} (1-\frac{1}{n})^{(v -2\overlap -w -1)_+}  \, .
\end{align}
Now there are at most $\overlap$ terms for each $e_w$. 
If $w \leq u - 3\overlap -1$, then the exponent is positive in every term and it is always bigger than $u -3\overlap -1 -w$, which means we can bound the sum by $\overlap (1-\frac{1}{n})^{u -3\overlap -1 -w}$. 
Otherwise we can simply bound the sum by $\overlap$. We get:
\begin{align}
\sum_{v=(u-\overlap)_+}^{u-1} H_v
\leq \sum_{w=1}^{u-2} 
	\big[
		\ind_{\{u -3\overlap \leq w \leq u -2\}}\overlap	
		+ \ind_{\{w \leq u -3\overlap -1\}}\overlap (1-\frac{1}{n})^{u -3\overlap -1 -w}
	\big] e_w  \, .
\end{align}

This means that for $w$ at time $t$:
\begin{align}\label{eq:r5}
r_5 
&\leq \frac{4\lipschitz\stepsize^2 C_2}{n} \sum_{u=0}^t (1 -\contraction)^{t-u}
	\big[
		\ind_{\{u -3\overlap \leq w \leq u -2\}}\overlap	
		+ \ind_{\{w \leq u -3\overlap -1\}}\overlap (1-\frac{1}{n})^{u -3\overlap -1 -w}
	\big]
\nonumber \\
&\leq \frac{4\lipschitz\stepsize^2 C_2}{n}
	\Big[
		\sum_{u=w+2}^{\min(t, w +3\overlap)} \overlap (1 -\contraction)^{t-u}
		+ \sum_{u=w +3\overlap +1}^t \overlap (1-\frac{1}{n})^{u -3\overlap -1 -w} (1 -\contraction)^{t-u}
	\Big]
\nonumber \\
&\leq \frac{4\lipschitz\stepsize^2 C_2}{n} \overlap
	\Big[
		(1 -\contraction)^{t-w}(1 -\contraction)^{-3\overlap} 3\overlap
\nonumber \\ &\qquad \qquad \qquad
		+ (1 -\contraction)^{t-w}(1 -\contraction)^{-1 -3\overlap}\sum_{u=w +3\overlap +1}^t (1-\frac{1}{n})^{u -3\overlap -1 -w} (1 -\contraction)^{-u +3\overlap +1 +w}
	\Big]
\nonumber \\
&\leq \frac{4\lipschitz\stepsize^2 C_2}{n} \overlap (1 -\contraction)^{t-w}(1 -\contraction)^{-3\overlap-1}
	\big(
		3\overlap
		+ \frac{1}{1-q}
	\big) \,  .
\end{align}

By combining the five terms together (\eqref{eq:r1},~\eqref{eq:r2},~\eqref{eq:r3},~\eqref{eq:r4} and~\eqref{eq:r5}), we get that $\forall u$ s.t. $1 \leq u \leq t$:
\begin{equation}\label{eq:rut}
\begin{aligned}
r_u^t \leq (1- \contraction)^{t-u} 
	\Big[
		&-2 \stepsize
		+4\lipschitz\stepsize^2 C_1
		+\frac{4\lipschitz\stepsize^2 C_1}{n}
			(1 - \contraction)^{-2\overlap -1} \big(
				2\overlap 
				+ \frac{1}{1-q}
			\big)
\\
		&+4\lipschitz\stepsize^2 C_2 \overlap (1-\contraction)^{-\overlap}
		+\frac{4\lipschitz\stepsize^2 C_2}{n} \overlap (1 -\contraction)^{-3\overlap -1}
			\big(
				3\overlap
				+ \frac{1}{1-q}
		\big)
	\Big] .
\end{aligned}
\end{equation}

\paragraph{Computation of $r_0^t$.}
Recall that we treat the $\tilde e_0$ term separately in Section~\ref{apxB:lma2}.
The initialization of \SAGA\ creates an initial synchronization, which means that the contribution of $\tilde e_0$ in our bound on $\E\|g_t\|^2$~\eqref{eq:53} is roughly $n$ times bigger than the contribution of any $e_u$ for $1 < u < t$.\footnote{This is explained in details right before~\eqref{eq:52}.}
In order to safely handle this term in our Lyapunov inequality, we only need to prove that it is bounded by a reasonable constant.
Here again, we split $r_0^t$ in five contributions coming from~\eqref{eq:rutMaster}:
\begin{itemize}
\item $r_1$, the part coming from the $-2\stepsize e_u$ terms;
\item $r_2$, coming from $4\lipschitz\stepsize^2 C_1 e_u$;
\item $r_3$, coming from $4\lipschitz\stepsize^2 C_1 (1 -\frac{1}{n})^{(u -\overlap)_+} \tilde e_0$;
\item $r_4$, coming from $4\lipschitz\stepsize^2 C_2\sum_{v=(u-\overlap)_+}^{u-1} e_v$;
\item $r_5$, coming from $4\lipschitz\stepsize^2 C_2\sum_{v=(u-\overlap)_+}^{u-1} (1 -\frac{1}{n})^{(v -\overlap)_+} \tilde e_0$.
\end{itemize}
Note that there is no $\tilde e_0$ in $H_t$, which is why we can safely ignore these terms here.

We have $r_1 = -2\stepsize (1 -\contraction)^t$ and $r_2=4\lipschitz\stepsize^2 C_1 (1 -\contraction)^t$.

Let us compute $r_3$.
\begin{align}
\sum_{u=0}^t (1 -\contraction)^{t -u} &(1 -\frac{1}{n})^{(u -\overlap)_+}
\nonumber \\
&= \sum_{u=0}^{\min(t, \overlap)} (1 -\contraction)^{t -u} 
	+ \sum_{u=\overlap +1}^t (1 -\contraction)^{t -u} (1 -\frac{1}{n})^{u -\overlap}	
\nonumber \\
&\leq (\overlap + 1) (1 -\contraction)^{t -\overlap} 
	+ (1 -\contraction)^{t -\overlap} \sum_{u=\overlap +1}^t (1 -\contraction)^{\overlap -u} (1 -\frac{1}{n})^{u -\overlap}	
\nonumber \\
&\leq (1 -\contraction)^t (1 -\contraction)^{-\overlap}
	\big(
		\overlap + 1 + \frac{1}{1 -q}
	\big)  \, .
\end{align}
This gives us:
\begin{align}
r_3 \leq (1 -\contraction)^t 4\lipschitz\stepsize^2 C_1 (1 -\contraction)^{-\overlap} \big(\overlap + 1 + \frac{1}{1 -q} \big)  \, .
\end{align}

We have already computed $r_4$ for $u>0$ and the computation is exactly the same for $u=0$. $r_4 \leq (1 - \contraction)^t 4\lipschitz\stepsize^2 C_2 \overlap(1-\contraction)^{-\overlap}$ .

Finally we compute $r_5$.
\begin{align}
\sum_{u=0}^t (1 -\contraction)^{t -u} &\sum_{v=(u -\overlap)_+}^{u -1} (1 -\frac{1}{n})^{(v -\overlap)_+}
\nonumber \\
&=\sum_{v=1}^{t -1} \sum_{u=v +1}^{\min(t, v +\overlap)} (1 -\contraction)^{t -u} (1 -\frac{1}{n})^{(v -\overlap)_+}
\nonumber \\
&\leq \sum_{v=1}^{\min(t-1, \overlap)} \sum_{u=v +1}^{v +\overlap} (1 -\contraction)^{t -u}
	+ \sum_{v=\overlap + 1}^{t -1} \sum_{u=v +1}^{\min(t, v +\overlap)} (1 -\contraction)^{t -u} (1 -\frac{1}{n})^{v -\overlap}
\nonumber \\
&\leq \overlap^2 (1 -\contraction)^{t -2\overlap}
	+ \sum_{v=\overlap +1}^{t -1} (1 -\frac{1}{n})^{v -\overlap} \overlap (1 -\contraction)^{t -v -\overlap} 
\nonumber \\
&\leq \overlap^2 (1 -\contraction)^{t -2\overlap}
	+ \overlap (1 -\contraction)^t (1 -\contraction)^{-2\overlap}\sum_{v=\overlap +1}^{t -1} (1 -\frac{1}{n})^{v -\overlap} \overlap (1 -\contraction)^{-v +\overlap} 
\nonumber  \\
&\leq (1 -\contraction)^t (1 -\contraction)^{-2\overlap}\big(\overlap^2 + \overlap \frac{1}{1 -q}\big)  \, .
\end{align}
Which means:
\begin{align}
r_5 \leq (1 -\contraction)^t 4\lipschitz\stepsize^2 C_2 (1 -\contraction)^{-2\overlap}\big(\overlap^2 + \overlap \frac{1}{1 -q}\big) .
\end{align}

Putting it all together, we get that: $\forall t \geq 0$
\begin{equation} \label{r0t}
\begin{aligned}
r_0^t \leq (1 -\contraction)^t 
	\Big[    \Big( &-2 \stepsize + 
		4\lipschitz\stepsize^2 C_1+4\lipschitz\stepsize^2 C_2 \overlap(1-\contraction)^{-\overlap} \Big) \frac{e_0}{\tilde e_0} \\
		&+4\lipschitz\stepsize^2 C_1 (1 -\contraction)^{-\overlap} \big(\overlap + 1 + \frac{1}{1 -q} \big) 
		+4\lipschitz\stepsize^2 C_2 \overlap (1 -\contraction)^{-2\overlap}\big(\overlap + \frac{1}{1 -q}\big)
	\Big] .
\end{aligned}
\end{equation}

\paragraph{Sufficient condition for convergence.}
We need all $r_u^t, u \geq 1$ to be negative so we can safely drop them from~\eqref{apx:Lyapunov}. 
Note that for every $u$, this is the same condition. 
We will reduce that condition to a second-order polynomial sign condition.
We also remark that since $\stepsize \geq 0$, we can upper bound our terms in $\stepsize$ and $\stepsize^2$ in this upcoming polynomial, which will give us sufficient conditions for convergence.

Now, as $\stepsize$ is part of $C_2$, we need to expand it once more to find our conditions. 
We have:
\begin{align*}
C_1 &= 1 + \sqrt{\sparsity}\overlap ; \qquad
C_2 = \sqrt{\sparsity} + \stepsize\strongconvex C_1  \, .
\end{align*}

Dividing the bracket in~\eqref{eq:rut} by~$\stepsize$ and rearranging as a second degree polynomial, we get the condition:
\begin{align} 
4\lipschitz &\Bigg(
	C_1 
	+ \frac{C_1}{n}(1-\contraction)^{-2\overlap -1}\Big[2\overlap + \frac{1}{1-q}\Big]
	+ \Big[\sqrt{\sparsity}\overlap(1-\contraction)^{-\overlap} + \frac{\sqrt{\sparsity}\overlap}{n}(1 -\contraction)^{-3\overlap -1}(3\overlap + \frac{1}{1 -q}) \Big]
	\Bigg) \stepsize
\nonumber \\
	&+ 8\strongconvex C_1 \lipschitz \overlap \Big[(1-\contraction)^{-\overlap} + \frac{1}{n}(1 -\contraction)^{-3\overlap -1}(3\overlap + \frac{1}{1 -q})
	\Big] \stepsize^2
	+ 2
\leq 0  \, . \label{eq:ConvergenceCondition}
\end{align}
The discriminant of this polynomial is always positive, so $\stepsize$ needs to be between its two roots. 
The smallest is negative, so the condition is not relevant to our case (where $\stepsize > 0$). 
By solving analytically for the positive root~$\phi$, we get an upper bound condition on~$\stepsize$ that can be used for any overlap~$\overlap$ and guarantee convergence. 
Unfortunately, for large~$\overlap$, the upper bound becomes exponentially small because of the presence of~$\overlap$ in the exponent in~\eqref{eq:ConvergenceCondition}. 
More specifically, by using the bound $1/(1-\contraction) \leq \exp(2\contraction)$\footnote{This bound can be derived from the inequality $(1-x/2) \geq \exp(-x)$ which is valid for $0 \leq x \leq 1.59$.} and thus $(1-\contraction)^{-\overlap} \leq \exp(2 \overlap \contraction)$ in~\eqref{eq:ConvergenceCondition}, we would obtain factors of the form $\exp(\overlap/n)$ in the denominator for the root~$\phi$ (recall that $\contraction < 1/n$). 

Our Lemma~\ref{lma:3} is derived instead under the assumption that $\overlap \leq \mathcal{O}(n)$, with the constants chosen in order to make the condition~\eqref{eq:ConvergenceCondition} more interpretable and to relate our convergence result with the standard SAGA convergence (see Theorem~\ref{th1}). As explained in Appendix~\ref{apxD}, the assumption that $\overlap \leq \mathcal{O}(n)$ appears reasonable in practice. 
First, by using Bernoulli's inequality, we have:
\begin{equation} \label{eq:Bernouilli}
(1 - \contraction)^{k \overlap} \geq 1 - k\overlap \contraction \qquad \textnormal{for integers} \quad k\overlap \geq 0 \, .
\end{equation}

To get manageable constants, we make the following slightly more restrictive assumptions on the target rate~$\contraction$\footnote{Note that we already expected $\contraction < 1/n$.} and overlap~$\overlap$:\footnote{This bound on $\overlap$ is reasonable in practice, see Appendix~\ref{apxD}.}
\begin{align}
\contraction &\leq \frac{1}{4n} \label{eq:assumptionContraction}
\\
\overlap &\leq \frac{n}{10} \, . \label{eq:assumptionOverlap}
\end{align}

We then have:
\begin{align}
\frac{1}{1-q} &\leq \frac{4n}{3} & &
\\
\frac{1}{1-\contraction} &\leq \frac{4}{3} & &
\\
k\overlap\contraction &\leq \frac{3}{40} & &
\text{for $1 \leq k \leq 3$}
\\
(1 -\contraction)^{-k\overlap} &\leq \frac{1}{1-k\overlap \contraction} \leq \frac{40}{37} & &
\text{for $1 \leq k \leq 3$ and by using~\eqref{eq:Bernouilli}}.
\end{align}

We can now upper bound loosely the three terms in brackets appearing in~\eqref{eq:ConvergenceCondition} as follows:
\begin{align}
(1-\contraction)^{-2\overlap -1} \big[2\overlap + \frac{1}{1-q}\big] &\leq 3n \label{eq:Simp1}
\\
\sqrt{\sparsity}\overlap(1-\contraction)^{-\overlap} + \frac{\sqrt{\sparsity}\overlap}{n}(1 -\contraction)^{-3\overlap -1}(3\overlap + \frac{1}{1 -q}) 
\leq 4 \sqrt{\sparsity}\overlap &\leq 4C_1
\\
(1-\contraction)^{-\overlap} + \frac{1}{n}(1 -\contraction)^{-3\overlap -1}(3\overlap + \frac{1}{1 -q}) &\leq 4 \, .
\label{eq:Simp2}
\end{align}

By plugging~\eqref{eq:Simp1}--\eqref{eq:Simp2} into~\eqref{eq:ConvergenceCondition}, we get the simpler sufficient condition on~$\stepsize$:
\begin{align}
-1 + 16\lipschitz C_1 \stepsize + 16\lipschitz C_1\strongconvex\overlap\stepsize^2 \leq 0 \, .
\end{align}

The positive root $\phi$ is:
\begin{align}\label{eq:phi}
\phi = \frac{16\lipschitz C_1(\sqrt{1 + \frac{\strongconvex\overlap}{4\lipschitz C_1}} -1)}{32 \lipschitz C_1\strongconvex\overlap}
= \frac{\sqrt{1 + \frac{\strongconvex\overlap}{4\lipschitz C_1}} -1}{2\strongconvex\overlap}  \, .
\end{align}

We simplify it further by using the inequality:\footnote{This inequality can be derived by using the concavity property $f(y) \leq f(x) + (y-x) f'(x)$ on the differentiable concave function $f(x)=\sqrt{x}$ with $y=1$.}
\begin{equation}\label{eq:concavesqrt}
\sqrt{x} - 1 \geq \frac{x - 1}{2 \sqrt{x}}  \qquad \forall x > 0 \, .
\end{equation}

Using~\eqref{eq:concavesqrt} in~\eqref{eq:phi}, and recalling that $\kappa := \lipschitz / \strongconvex$, we get:
\begin{align}
\phi \geq \frac{1}{16\lipschitz C_1 \sqrt{1 + \frac{\overlap}{4\kappa C_1}}} \, .
\end{align}

Since $\frac{\overlap}{C_1} = \frac{\overlap}{1 + \sqrt{\sparsity}\overlap} \leq \min(\overlap, \frac{1}{\sqrt{\sparsity}})$, we get that a sufficient condition on our stepsize is:
\begin{equation}
\stepsize \leq \frac{1}{16\lipschitz (1 + \sqrt{\sparsity} \overlap) \sqrt{1 + \frac{1}{4\kappa} \min(\overlap, \frac{1}{\sqrt{\sparsity}})}}  \, .
\end{equation}

Subject to our conditions on $\stepsize$, $\contraction$ and $\overlap$, we then have that: $r_u^t \leq 0\ \text{for all}\ u\ \text{s.t.}\ 1 \leq u \leq t$. 
This means we can rewrite~\eqref{apx:Lyapunov} as:
\begin{align}\label{eq:lya2}
\lyapunov_{t+1} \leq (1 - \contraction)^{t+1}a_0 + (1-\frac{\stepsize\strongconvex}{2})\lyapunov_t + r_0^t \tilde e_0  \, .
\end{align}

Now, we could finish the proof from this inequality, but it would only give us a convergence result in terms of $a_t = \E \|x_t - x^*\|^2$. A better result would be in terms of the suboptimality at $\hat x_t$ (because $\hat x_t$ is a real quantity in the algorithm whereas $x_t$ is virtual). Fortunately, to get such a result, we can easily adapt~\eqref{eq:lya2}.

We make $e_t$ appear on the left side of~\eqref{eq:lya2}, by adding $\stepsize$ to $r_t^t$ in~\eqref{apx:Lyapunov}:\footnote{We could use any multiplier from $0$ to $2\stepsize$, but choose $\stepsize$ for simplicity. For this reason and because our analysis of the $r_t^t$ term was loose, we could derive a tighter bound, but it does not change the leading terms.}
\begin{align}\label{eq:newlyapunov}
\stepsize e_t + \lyapunov_{t+1} \leq (1 - \contraction)^{t+1}a_0 + (1-\frac{\stepsize\strongconvex}{2})\lyapunov_t + \sum_{u=1}^{t-1} r_u^t e_u + r_0^t \tilde e_0 + (r_t^t + \stepsize)e_t .
\end{align}

We now require the stronger property that $\stepsize +r_t^t \leq 0$, which translates to replacing $-2\stepsize$ with $-\stepsize$ in~\eqref{eq:rut}:
\begin{equation}
\begin{aligned}
0 \geq 
	\Big[
		&- \stepsize
		+4\lipschitz\stepsize^2 C_1
		+\frac{4\lipschitz\stepsize^2 C_1}{n}
			(1 - \contraction)^{-2\overlap - 1} \big(
				2\overlap
				+\frac{1}{1-q}
			\big) 
\\
		&+4\lipschitz\stepsize^2  C_2 \overlap(1-\contraction)^{-\overlap}
		+\frac{4\lipschitz\stepsize^2  C_2}{n} \overlap (1 -\contraction)^{-3\overlap-1}
			\big(
				3\overlap
				+ \frac{1}{1-q}
		\big)
	\Big] .
\end{aligned}
\end{equation}

We can easily derive a new stronger condition on $\stepsize$ under which we can drop all the $e_u, u > 0$ terms in~\eqref{eq:newlyapunov}:
\begin{align}
\stepsize \leq \stepsize^* = \frac{1}{32\lipschitz (1 + \sqrt{\sparsity} \overlap) \sqrt{1 + \frac{1}{8\kappa} \min(\overlap, \frac{1}{\sqrt{\sparsity}})}},
\end{align}
and thus under which we get:
\begin{align}\label{eq:lyapu3}
\stepsize  e_t + \lyapunov_{t+1} \leq (1 - \contraction)^{t+1}a_0 + (1-\frac{\stepsize\strongconvex}{2})\lyapunov_t + r_0^t \tilde e_0 .
\end{align}

This finishes the proof of Lemma~\ref{lma:3}. 
\qed

\subsection{Proof of Theorem~\ref{thm:convergence}}\label{apxB:th2}
\paragraph{End of Lyapunov convergence.}
We continue with the assumptions of Lemma~\ref{lma:3} which gave us~\eqref{eq:lyapu3}.
Thanks to~\eqref{r0t}, we can also rewrite $r_0^t \leq (1 -\contraction)^{t+1} A$, where $A$ is a constant which depends on $n$, $\sparsity$, $\stepsize$ and $\lipschitz$ but is finite and crucially does not depend on $t$.
In fact, by reusing similar arguments as in~\ref{apxB:lma3}, we can show the bound $A \leq \stepsize n$ under the assumptions of Lemma~\ref{lma:3} (including $\stepsize \leq \stepsize^*$).\footnote{In particular, note that $e_0$ does not appear in the definition of $A$ because it turns out that the parenthesis group multiplying $e_0$ in~\eqref{r0t} is negative. Indeed, it contains less positive terms than~\eqref{eq:rut} which we showed to be negative under the assumptions from Lemma~\ref{lma:3}.}
We then have:
\begin{align}
\lyapunov_{t+1}  \leq \stepsize e_t + \lyapunov_{t+1} 
&\leq (1-\frac{\stepsize\strongconvex}{2})\lyapunov_t  +(1 - \contraction)^{t+1} (a_0 + A \tilde e_0) 
\nonumber \\
&\leq (1-\frac{\stepsize\strongconvex}{2})^{t+1}\lyapunov_0 + (a_0 + A \tilde e_0) \sum_{k=0}^{t+1} (1 - \contraction)^{t+1 -k} (1-\frac{\stepsize\strongconvex}{2})^k .
\end{align}

We have two linearly contracting terms. 
The sum contracts linearly with the minimum geometric rate factor between $\stepsize \strongconvex / 2$ and $\contraction$. 
If we define $m := \min(\contraction, \stepsize\strongconvex/2)$, $M := \max(\contraction, \stepsize\strongconvex/2)$ and $\contraction^* := \nu m$ with $0 <\nu < 1$,\footnote{$\nu$ is introduced to circumvent the problematic case where $\contraction$ and $\stepsize \strongconvex / 2$ are too close together, which does not prevent the geometric convergence, but makes the constant $\frac{1}{1-\eta}$ potentially very big (in the case both terms are equal, the sum even becomes an annoying linear term in t).} we then get:\footnote{Note that if $m \neq \contraction$, we can perform the index change $t+1-k \rightarrow k$ to get the sum.}
\begin{align}
\stepsize e_t \leq 
\stepsize e_t + \lyapunov_{t+1} 
&\leq (1-\frac{\stepsize\strongconvex}{2})^{t+1}\lyapunov_0 + (a_0 + A \tilde e_0) \sum_{k=0}^{t+1} (1 - m)^{t+1-k} (1-M)^k
\nonumber \\
&\leq (1-\frac{\stepsize\strongconvex}{2})^{t+1}\lyapunov_0 + (a_0 + A \tilde e_0) \sum_{k=0}^{t+1} (1 - \contraction^*)^{t+1-k} (1-M)^k
\nonumber \\
&\leq (1-\frac{\stepsize\strongconvex}{2})^{t+1}\lyapunov_0 + (a_0 + A \tilde e_0) (1 - \contraction^*)^{t+1} \sum_{k=0}^{t+1} (1 - \contraction^*)^{-k} (1-M)^k
\nonumber \\
&\leq (1-\frac{\stepsize\strongconvex}{2})^{t+1}\lyapunov_0 + (1 - \contraction^*)^{t+1} \frac{1}{1 -\eta} (a_0 + A \tilde e_0)
\nonumber \\
&\leq (1 - \contraction^*)^{t+1} \big(a_0 + \frac{1}{1 -\eta} (a_0 + A \tilde e_0) \big) ,
\end{align}
where $\eta := \frac{1-M}{1-\contraction^*}$. 
We have $\frac{1}{1 - \eta} = \frac{1 - \contraction^*}{M - \contraction^*}$.

By taking $\nu = \frac{4}{5}$ and setting $\contraction = \frac{1}{4n}$ -- its maximal value allowed by the assumptions of Lemma~\ref{lma:3} -- we get $M \geq \frac{1}{4n}$ and $\contraction^* \leq \frac{1}{5n}$, which means $\frac{1}{1 - \eta} \leq 20n$.

All told, using $A \leq \stepsize n$, we get:
\begin{equation}
e_t \leq (1 - \contraction^*)^{t+1} \tilde C_0,
\end{equation}
where
\begin{equation}
\tilde C_0 := \frac{21n}{\stepsize}\Big(\|x_0 - x^*\|^2 + \stepsize \frac{n}{2L}\E\|\alpha_i^0 - f'_i(x^*)\|^2 \Big) \, .
\end{equation}

Since we set $\contraction = \frac{1}{4n}, \nu = \frac{4}{5}$, we have $\nu \contraction = \frac{1}{5n}$.
Using a stepsize $\stepsize = \frac{a}{\lipschitz}$ as in Theorem~\ref{thm:convergence}, we get $\nu \frac{\stepsize \strongconvex}{2} = \frac{2a}{5 \kappa}$.
We thus obtain a geometric rate of $\contraction^* = \min \{\frac{1}{5n}, a\frac{2}{5\kappa}\}$, which we simplified to $\frac{1}{5} \min \{\frac{1}{n}, a\frac{1}{\kappa}\}$ in Theorem~\ref{thm:convergence}, finishing the proof. We also observe that $\tilde C_0 \leq \frac{60n}{\stepsize} C_0$, with $C_0$ defined in Theorem~\ref{th1}.
\qed

\subsection{Proof of Corollary~\ref{thm:bigdata} (speedup regimes)}
Referring to~\citet{qsaga} and our own Theorem~\ref{th1}, the geometric rate factor of \SAGA\ is $\frac{1}{5}\min\{\frac{1}{n}, \frac{a}{\kappa}\}$ for a stepsize of $\stepsize = \frac{a}{5L}$. We start by proving the first part of the corollary which considers the step size $\stepsize = \frac{a}{L}$ with $a = a^*(\overlap)$.
We distinguish between two regimes to study the parallel speedup our algorithm obtains and to derive a condition on $\overlap$ for which we have a linear speedup.

\paragraph{Big Data.} 
In this regime, $n > \kappa$ and the geometric rate factor of sequential \SAGA\ is $\frac{1}{5n}$. 
To get a linear speedup (up to a constant factor), we need to enforce $\contraction^* = \Omega(\frac{1}{n})$.
We recall that $\contraction^* = \min \{\frac{1}{5n}, a\frac{1}{5\kappa}\}$.

We already have $\frac{1}{5n} = \Omega(\frac{1}{n})$. This means that we need $\overlap$ to verify $\frac{a^*(\overlap)}{5\kappa} = \Omega(\frac{1}{n})$, where $a^*(\overlap) = \frac{1}{32 \left(1+ \overlap  \sqrt \sparsity \right) \xi(\kappa, \sparsity, \overlap)}$ according to Theorem~\ref{thm:convergence}.
Recall that $\xi(\kappa, \sparsity, \overlap) := \sqrt{1 + \frac{1}{8 \kappa}  \min\{\frac{1}{\sqrt{\sparsity}}, \overlap\} }$.
Up to a constant factor, this means we can give the following sufficient condition:

\begin{equation}
\frac{1}{\kappa \left(1+ \overlap  \sqrt \sparsity \right) \xi(\kappa, \sparsity, \overlap)}
= \Omega \Big(\frac{1}{n}\Big)
\end{equation}
i.e. 
\begin{equation} \label{eq:SufficientCondition}
\left(1+ \overlap  \sqrt \sparsity \right) \xi(\kappa, \sparsity, \overlap)
= \mathcal{O} \Big( \frac{n}{\kappa} \Big) \, .
\end{equation}

We now consider two alternatives, depending on whether $\kappa$ is bigger than $\frac{1}{\sqrt{\sparsity}}$ or not. If $\kappa \geq \frac{1}{\sqrt{\sparsity}}$, then $\xi(\kappa, \sparsity, \overlap) < 2$ and we can rewrite the sufficient condition~\eqref{eq:SufficientCondition} as:

\begin{align}
\overlap = \mathcal{O}(1) \frac{n}{\kappa\sqrt{\sparsity}}.
\end{align}

In the alternative case, $\kappa \leq \frac{1}{\sqrt{\sparsity}}$. 
Since $a^*(\overlap)$ is decreasing in $\overlap$, we can suppose $\overlap \geq \frac{1}{\sqrt{\sparsity}}$ without loss of generality and thus $\xi(\kappa, \sparsity, \overlap) = \sqrt{1 + \frac{1}{8 \kappa \sqrt{\sparsity}}}$.
We can then rewrite the sufficient condition~\eqref{eq:SufficientCondition} as:
\begin{align}\label{eq:sufcondcor3}
\frac{\overlap \sqrt{\sparsity}}{\sqrt{\kappa}\sqrt[4]{\sparsity}} &= \mathcal{O}(\frac{n}{\kappa})
\nonumber \\
\overlap &= \mathcal{O}(1)\frac{n}{\sqrt{\kappa}\sqrt[4]{\sparsity}} \, .
\end{align}

We observe that since we have supposed that $\kappa \leq \frac{1}{\sqrt{\sparsity}}$, we have $\sqrt{\kappa \sqrt{\sparsity}} \leq \kappa \sqrt{\sparsity} \leq 1$, which means that our initial assumption that $\overlap < \frac{n}{10}$ is stronger than condition~\eqref{eq:sufcondcor3}.

We can now combine both cases to get the following sufficient condition for the geometric rate factor of \ASAGA\ to be the same order as sequential \SAGA\ when $n > \kappa$:
\begin{align}
\overlap = \mathcal{O}(1) \frac{n}{\kappa\sqrt{\sparsity}}; \quad
\overlap = \mathcal{O}(n) \, . 
\end{align}

\paragraph{Ill-conditioned regime.}
In this regime, $\kappa > n$ and the geometric rate factor of sequential \SAGA\ is $a \frac{1}{\kappa}$. 
Here, to obtain a linear speedup, we need $\contraction^* = \mathcal{O}(\frac{1}{\kappa})$.
Since $\frac{1}{n} > \frac{1}{\kappa}$, all we require is that $\frac{a^*(\overlap)}{\kappa} = \Omega(\frac{1}{\kappa})$ where $a^*(\overlap) = \frac{1}{32 \left(1+ \overlap  \sqrt \sparsity \right) \xi(\kappa, \sparsity, \overlap)}$, which reduces to $a^*(\overlap) = \Omega(1)$.

We can give the following sufficient condition:
\begin{align}
\frac{1}{\left(1+ \overlap  \sqrt \sparsity \right) \xi(\kappa, \sparsity, \overlap)} = \Omega(1)
\end{align}

Using that $\frac{1}{n} \leq \sparsity \leq 1$ and that $\kappa > n$, we get that $\xi(\kappa, \sparsity, \overlap) \leq 2$, which means our sufficient condition becomes:

\begin{align}
\overlap \sqrt{\sparsity} &= \mathcal{O}(1)
\nonumber \\
\overlap &= \frac{\mathcal{O}(1)}{\sqrt{\sparsity}} .
\end{align}
This finishes the proof for the first part of Corollary~\ref{thm:illcondition}.	

\paragraph{Universal stepsize.} If $\overlap = \mathcal{O}(\frac{1}{\sqrt{\sparsity}})$, then $\xi(\kappa, \sparsity, \overlap) = \mathcal{O}(1)$ and $(1+\overlap \sqrt{\sparsity}) = \mathcal{O}(1)$, and thus $a^*(\overlap) = \Omega(1)$ (for any $n$ and $\kappa$). This means that the universal stepsize $\stepsize = \Theta(1/L)$ satisfies $\stepsize \leq a^*(\overlap)$ for any $\kappa$, giving the same rate factor $\Omega( \min\{\frac{1}{n}, \frac{1}{\kappa}\})$ that sequential \SAGA\ has, completing the proof for the second part of Corollary~\ref{thm:illcondition}. 
\qed

\section{Additional experimental results}\label{apx:AER}
\subsection{Effect of sparsity}
Sparsity plays an important role in our theoretical results, where we find that while it is necessary in the ``ill-conditioned'' regime to get linear speedups, it is not in the ``well-conditioned'' regime.
We confront this to real-life experiments by comparing the convergence and speedup performance of our three asynchronous algorithms on the Covtype dataset, which is fully dense after standardization.
The results appear in Figure~\ref{fig:covtype}.

While we still see a significant improvement in speed when increasing the number of cores, this improvement is smaller than the one we observe for sparser datasets.
The speedups we observe are consequently smaller, and taper off earlier than on our other datasets.
However, since the observed ``theoretical'' speedup is linear (see Section~\ref{apx:speedup}), we can attribute this worse performance to higher hardware overhead.
This is expected because each update is fully dense and thus the shared parameters are much more heavily contended for than in our sparse datasets.

\begin{figure}
\center \includegraphics[width=0.7\linewidth]{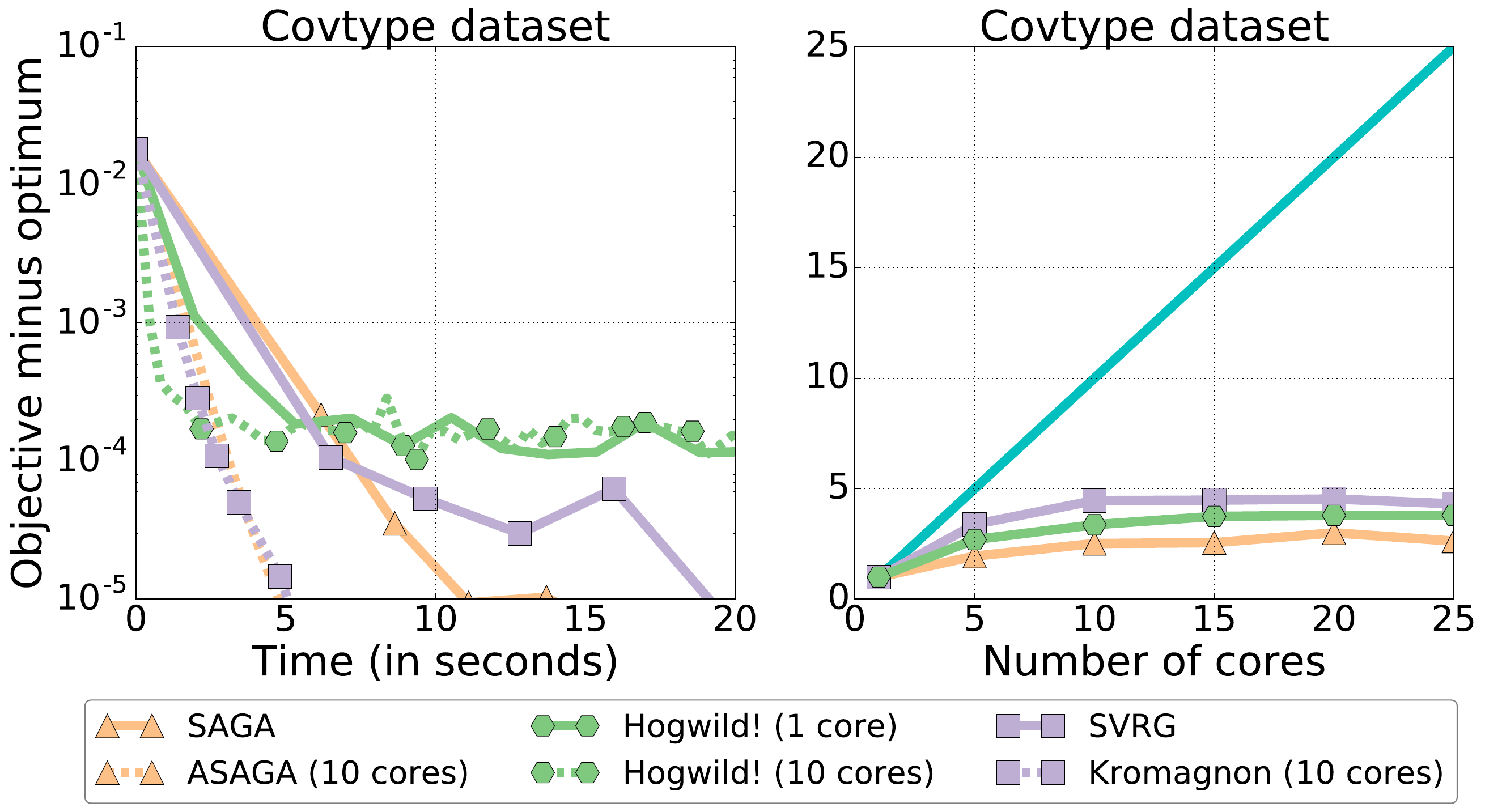}
\caption{Comparison on the Covtype dataset. Left: suboptimality. Right: speedup. The number of cores in the legend only refers to the left plot. }\label{fig:covtype}
\end{figure}

One thing we notice when computing the $\sparsity$ variable for our datasets is that it often fails to capture the full sparsity distribution, being essentially a maximum.
This means that $\sparsity$ can be quite big even for very sparse datasets.
Deriving a less coarse bound remains an open problem.

\subsection{Theoretical speedups}\label{apx:speedup}
In the main text of this paper, we show experimental speedup results where suboptimality is a function of the running time.
This measure encompasses both theoretical algorithmic properties and hardware overheads (such as contention of shared memory) which are not taken into account in our analysis.

In order to isolate these two effects, we plot our convergence experiments where suboptimality is a function of the number of iterations; thus, we abstract away any potential hardware overhead.\footnote{To do so, we implement a global counter which is sparsely updated (every $100$ iterations for example) in order not to modify the asynchrony of the system. 
This counter is used only for plotting purposes and is not needed otherwise.
}
The experimental results can be seen in Figure~\ref{fig:theoretical_speedups}. 

\begin{figure*}
\includegraphics[width=\linewidth]{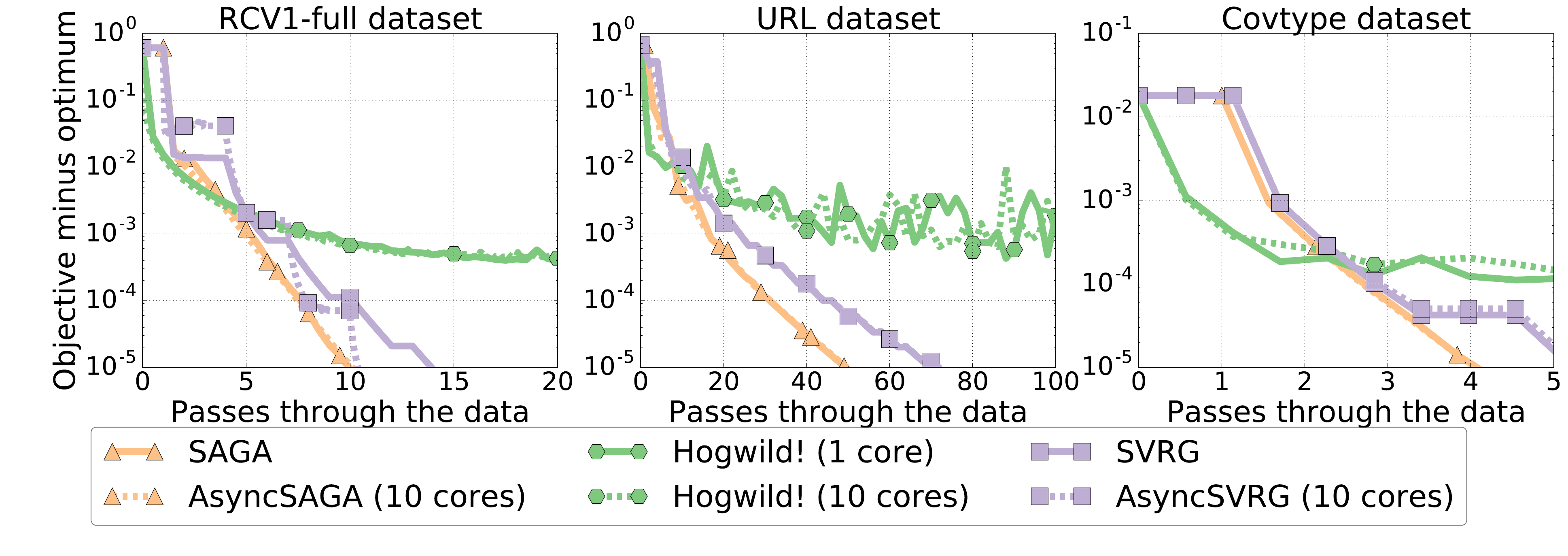}
\caption{{\bf Theoretical speedups}. 
Suboptimality with respect to number of iterations for \ASAGA\, \SVRG\ and \Hogwild\ with 1 and 10 cores.  
Curves almost coincide, which means the theoretical speedup is almost the number of cores $p$, hence linear.}\label{fig:theoretical_speedups}
\end{figure*} 

For all three algorithms and all three datasets, the curves for $1$ and $10$ cores almost coincide, which means that we are indeed in the ``theoretical linear speedup'' regime.
Indeed, when we plotted the amount of iterations required to converge to a given accuracy as a function of the number of cores, we obtained straight  horizontal lines for our three algorithms.
 
The fact that the speedups we observe in running time are less than linear can thus be attributed to various hardware overheads, including shared variable contention -- the compare-and-swap operations are more and more expensive as the number of competing requests increases -- and cache effects as mentioned in Section~\ref{ssec:results}.

\section{A closer look at the $\overlap$ constant} \label{apxD}
\subsection{Theory}
In the parallel optimization literature, $\overlap$ is often referred to as a proxy for the number of cores.
However, intuitively as well as in practice, it appears that there are a number of other factors that can influence this quantity.
We will now attempt to give a few qualitative arguments as to what these other factors might be and how they relate to $\overlap$.

\paragraph{Number of cores.}
The first of these factors is indeed the number of cores. 

If we have $p$ cores, $\overlap \geq p - 1$.
Indeed, in the best-case scenario where all cores have exactly the same execution speed for a single iteration, $\overlap = p -1$.

To get more insight into what $\overlap$ really encompasses, let us now try to define the worst-case scenario in the preceding example.
Consider $2$ cores. 
In the worst case scenario, one core runs while the other is stuck. 
Then the overlap is $t$ for all $t$ and eventually grows to $+\infty$. 
If we assume that one core runs twice as fast as the other, then $\overlap = 2$. 
If both run at the same speed, $\overlap = 1$.

It appears then that a relevant quantity is $R$, the ratio between the fastest execution time to the slowest execution time for a single iteration. 
$\overlap \leq (p-1) R$, which can be arbitrarily bigger than $p$.

\paragraph{Length of an iteration.} 
There are several factors at play in $R$ itself. 
\begin{itemize}
\item The first is the speed of execution of the cores themselves (i.e. clock time). 
The dependency here is quite clear.

\item The second is the data matrix itself. If one $f_i$ has support of size $n$ while all the others have support of size $1$, $r$ may eventually become very big.

\item The third is the length of the computation itself. The longer our algorithm runs, the more likely it is to explore the potential corner cases of the data matrix.
\end{itemize}

The overlap is upper bounded by the number of cores times the maximum iteration time over the minimum iteration time (which is linked to the sparsity distribution of the data matrix).
This is an upper bound, which means that in some cases it will not really be useful. 
For example, in the case where one factor has support size $1$ and all others have support size $d$, the probability of the event which corresponds to the upper bound is exponentially small in $d$. 
We conjecture that a more useful indicator could be the maximum iteration time over the expected iteration time.

To sum up this preliminary theoretical exploration, the $\overlap$ term encompasses a lot more complexity than is usually implied in the literature. 
This is reflected in the experiments we ran, where the constant was orders of magnitude bigger than the number of cores.

\subsection{Experimental results}
In order to verify our intuition about the $\overlap$ variable, we ran several experiments on all three datasets, whose characteristics are reminded in Table~\ref{table:1}.
$\delta_l^i$ is the support size of $f_i$.

\begin{table}[ht]
\caption{Density measures including minimum, average and maximum support size $\delta_l^i$ of the factors.}
\centering
\label{dataset-table2}
\begin{tabular}{lccccccc} 
\toprule
{} & $n$ & $d$ & density & $\max(\delta_l^i)$ & $\min(\delta_l^i)$ & $\bar \delta_l$ & $\max(\delta_l^i) / \bar \delta_l$\\
\midrule
{\bf RCV1} & \hfill 697,641 & \hfill 47,236 & \hfill 0.15\% & \hfill 1,224 & \hfill 4 & \hfill 73.2 & \hfill 16.7\\ 
{\bf URL} & \hfill 2,396,130 & \hfill 3,231,961 & \hfill 0.003\% & \hfill 414 & \hfill 16 & \hfill 115.6 & \hfill 3.58 \\
{\bf Covtype} & \hfill 581,012 & \hfill 54 & \hfill 100\% & \hfill 12 & \hfill 8 & \hfill 11.88 & \hfill 1.01 \\
\bottomrule
\end{tabular}
\label{table:1}
\end{table}

To estimate $\overlap$, we compute the average overlap over $100$ iterations, which is a lower bound on the actual overlap (which is a maximum, not an average). 
We then take the maximum observed quantity.
We use an average because computing the overlap requires using a global counter, which we do not want to update every iteration since it would make it a heavily contentious quantity susceptible of artificially changing the asynchrony of our algorithm.

The results we observe are order of magnitude bigger than $p$, indicating that $\overlap$ can indeed not be dismissed as a mere proxy for the number of cores, but has to be more carefully analyzed.

First, we plot the maximum observed $\overlap$ as a function of the number of cores (see Figure~\ref{fig:overlap}).
We observe that the relationship does indeed seem to be roughly linear with respect to the number of cores until 30 cores.
After 30 cores, we observe what may be a phase transition where the slope increases significantly.

\begin{figure*}
\includegraphics[width=\linewidth]{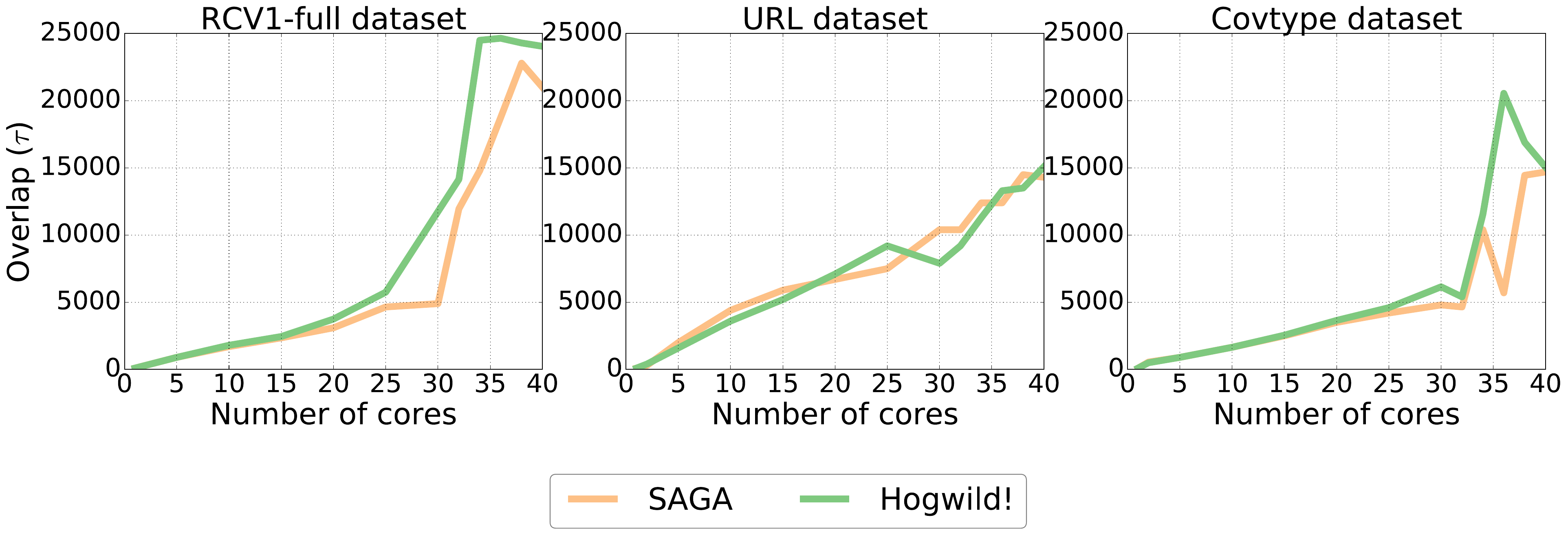}
\caption{{\bf Overlap}. 
Overlap as a function of the number of cores for both \ASAGA\ and \Hogwild\ on all three datasets.}\label{fig:overlap}
\end{figure*} 

Second, we measured the maximum observed $\overlap$ as a function of the number of epochs.
We omit the figure since we did not observe any dependency; that is, $\overlap$ does not seem to depend on the number of epochs.
We know that it must depend on the number of iterations (since it cannot be bigger, and is an increasing function with respect to that number for example), but it appears that a stable value is reached quite quickly (before one full epoch is done).

If we allowed the computations to run forever, we would eventually observe an event such that $\overlap$ would reach the upper bound mentioned in the last section, so it may be that $\overlap$ is actually a very slowly increasing function of the number of iterations.

\section{Lagged updates and Sparsity}\label{apxC}

\subsection{Comparison with Lagged Updates in the sequential case}
The lagged updates technique in \SAGA\ is based on the observation that the updates for component $[x]_v$ need not be applied until this coefficient needs to be accessed, that is, until the next iteration $t$ such that $ v \in S_{i_t}$. 
We refer the reader to~\citet{laggedsaga} for more details.

Interestingly, the expected number of iterations between two steps where a given dimension $v$ is involved in the partial gradient is $p_v^{-1}$, where $p_v$ is the probability that $v$ is involved in a given step. 
$p_v^{-1}$ is precisely the term which we use to multiply the update to $[x]_v$ in Sparse \SAGA. 
Therefore one may see the updates in Sparse \SAGA\ as \textit{anticipated} updates, whereas those in the~\citet{laggedsaga} implementation are \textit{lagged}. 
The two algorithms appear to be very close, even though Sparse \SAGA\ uses an expectation to multiply a given update whereas the lazy implementation uses a random variable (with the same expectation). Sparse \SAGA\ therefore uses a slightly more aggressive strategy, which gave faster run-time in our experiments below.

Although Sparse \SAGA\ requires the computation of the $p_v$ probabilities, this can be done during a first pass throughout the data (during which constant step size \SGD\ may be used) at a negligible cost. 

\begin{figure*}
\includegraphics[width=\linewidth]{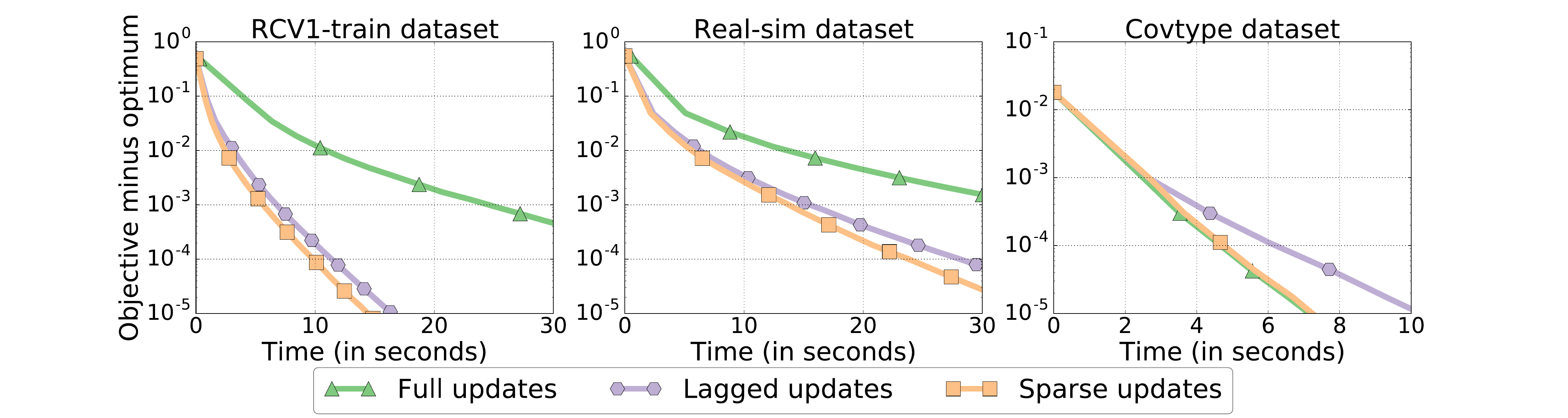}
\includegraphics[width=\linewidth]{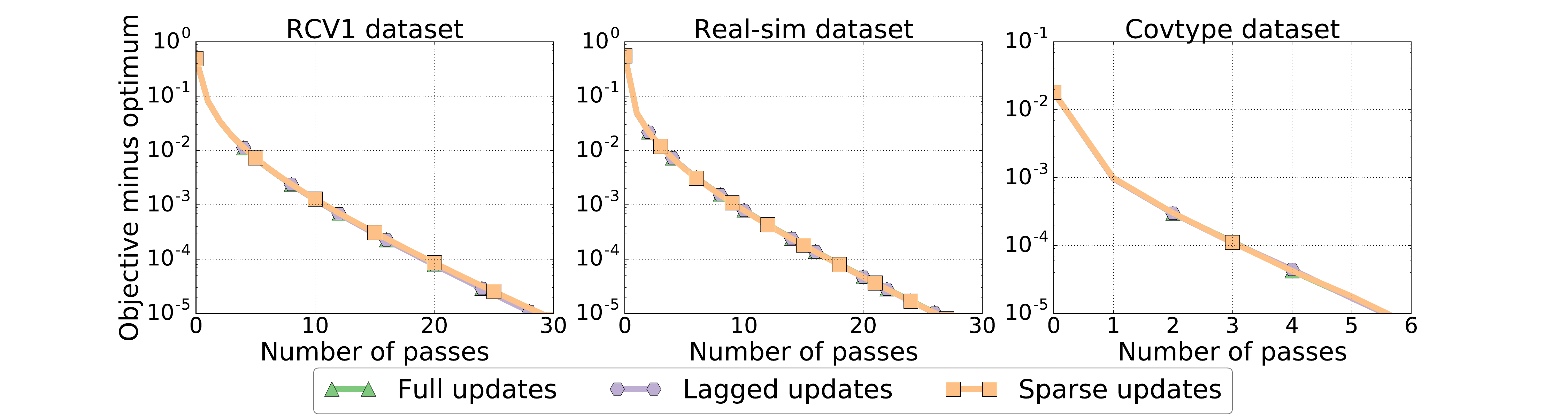}
\caption{{\bf Lagged vs sparse \SAGA\ updates}. 
Suboptimality with respect to time for different \SAGA\ update schemes on various datasets. First row: suboptimality as a function of time. Second row: suboptimality as a the number of passes over the dataset.
For sparse datasets (RCV1 and Real-sim), lagged and sparse updates have a lower cost per iteration which result in faster convergence.}\label{fig:fig_1}
\end{figure*} 

In our experiments, we compare the Sparse \SAGA\ variant proposed in Section~\ref{scs:sparse_saga} to two other approaches: the naive (i.e. dense) update scheme and the lagged updates implementation described in~\citet{SAGA}.
Note that we use different datasets from the parallel experiments, including a subset of the RCV1 dataset and the realsim dataset. 
Figure~\ref{fig:fig_1} reveals that sparse and lagged updates have a lower cost per iteration, resulting in faster convergence for sparse datasets.
Furthermore, while the two approaches had similar convergence in terms of number of iterations, the Sparse \SAGA\ scheme is slightly faster in terms of runtime (and as previously pointed out, sparse updates are better adapted for the asynchronous setting).
For the dense dataset (Covtype), the three approaches exhibit a similar performance. 

\subsection{On the difficulty of parallel lagged updates} \label{apx:DifficultyLagged}

In the implementation presented in~\citet{laggedsaga}, the dense part ($\bar \alpha$) of the updates is deferred. 
Instead of writing dense updates, counters $c_d$ are kept for each coordinate of the parameter vector -- which represent the last time these variables were updated -- as well as the average gradient $\bar \alpha$ for each coordinate. 
Then, whenever a component $[\hat x]_d$ is needed (in order to compute a new gradient), we subtract $\stepsize (t-c_d) [\bar \alpha]_d$ from it and $c_d$ is set to $t$.
The reason we can do this without modifying the algorithm is that $[\bar \alpha]_d$ only changes when $[\hat x]_d$ also does.

In the sequential setting, this is strictly the same as doing the updates in a dense way, since the coordinates are only stale when they're not used. 
Note that at the end of an execution all counters have to be subtracted at once to get the true final parameter vector (and to bring every $c_d$ counter to the final $t$).

In the parallel setting, several issues arise:
\begin{itemize}
\item two cores might be attempting to correct the lag at the same time. 
In which case since updates are done as additions and not replacements (which is necessary to ensure that there are no overwrites), the lag might be corrected multiple times, i.e. overly corrected. 
\item we would have to read and write atomically to each $[\hat x_d], c_d, [\bar \alpha]_d$ triplet, which is highly impractical.
\item we would need to have an explicit global counter, which we do not in \ASAGA\ (our global counter~$t$ being used solely for the proof).
\item in the dense setting, updates happen coordinate by coordinate. 
So at time $t$ the number of $\bar \alpha$ updates a coordinate has received from a fixed past time $c_d$ is a random variable, which may differs from coordinate to coordinate. 
Whereas in the lagged implementation, the multiplier is always $(t-c_d)$ which is a constant (conditional to $c_d$), which means a potentially different $\hat x_t$.
\end{itemize}

All these points mean both that the implementation of such a scheme in the parallel setting would be impractical, and that it would actually yields a different algorithm than the dense version, which would be even harder to analyze.

\section{Additional empirical details} \label{apxE}
\subsection{Detailed description of datasets}
We run our experiments on four datasets. In every case, we run logistic regression for the purpose of binary classification.

\paragraph{RCV1 ($n=697,641$, $d=47,236$).}
The first is the Reuters Corpus Volume I (RCV1) dataset~\citep{RCV1}, an archive of over 800,000 manually categorized newswire stories made available by Reuters, Ltd. for research purposes. 
The associated task is a binary text categorization.

\paragraph{URL ($n=2,396,130$, $d=3,231,961$).}
Our second dataset was first introduced in~\citet{URL}. Its associated task is a binary malicious url detection.
This dataset contains more than 2 million URLs obtained at random from Yahoo's directory listing (for the ``benign'' URLs) and from a large Web mail provider (for the ``malicious'' URLs).
The benign to malicious ratio is 2.
Features include lexical information as well as metadata.
This dataset was obtained from the libsvmtools project.\footnote{\url{http://www.csie.ntu.edu.tw/~cjlin/libsvmtools/datasets/binary.html}}

\paragraph{Covertype ($n=581,012$, $d=54$).}
On our third dataset, the associated task is a binary classification problem (down from 7 classes originally, following the pre-treatment of~\citet{Covtype}). The features are cartographic variables.
Contrarily to the first two, this is a dense dataset.

\paragraph{Realsim ($n=73,218$, $d=20,958$).}
We only use our fourth dataset for non-parallel experiments and a specific compare-and-swap test.
It constitutes of UseNet articles taken from four discussion groups (simulated auto racing, simulated aviation, real autos, real aviation).

\subsection{Implementation details}
\paragraph{Hardware.} 
All experiments were run on a Dell PowerEdge 920 machine with 4 Intel Xeon E7-4830v2 processors with 10 2.2GHz cores each and 384GB 1600 Mhz RAM.

\paragraph{Software.} \label{scalavsc}
All algorithms were implemented in the Scala language and the software stack consisted of a Linux operating system running Scala 2.11.7 and Java 1.6. 

We chose this expressive, high level language for our experimentation despite its typical 20x slower performance compared to C because our primary concern was that the code may easily be reused and extended for research purposes (which is harder to achieve with low level, heavily optimized C code; especially for error prone parallel computing). 

As a result our timed experiments exhibit sub-optimal running times, e.g. compared to~\citet{KR13}.
This is as we expected.
The observed slowdown is both consistent across datasets (roughly 20x) and with other papers that use Scala code (e.g.~\citet{mania}, ~\citet[Fig. 2]{cocoa}).

Despite this slowdown, our experiments show state-of-the-art results in convergence per number of iterations.
Furthermore, the speed-up patterns that we observe for our implementation of Hogwild and Kromagnon are similar to the ones given in [MN15], Niu et al.[2011] and Reddi et al.[2015] (in various languages).

The code we used to run all the experiments is available at \url{https://github.com/RemiLeblond/ASAGA}.

\paragraph{Necessity of compare-and-swap operations.}
Interestingly, we have found necessary to use compare-and-swap instructions in the implementation of \ASAGA. 
In Figure~\ref{fig:cas_comparison}, we display suboptimality plots using non-thread safe operations and compare-and-swap (CAS) operations. The non-thread safe version starts faster but then fails to converge beyond a specific level of suboptimality, while the compare-and-swap version does converges linearly up to machine precision.

For \textit{compare-and-swap} instructions we used the \texttt{AtomicDoubleArray} class from the Google library \texttt{Guava}. This class uses an \texttt{AtomicLongArray} under the hood (from package \texttt{java.util.concurrent.atomic} in the standard Java library), which does indeed benefit from lower-level CPU-optimized instructions.

\begin{figure}
\center \includegraphics[width=0.5 \linewidth]{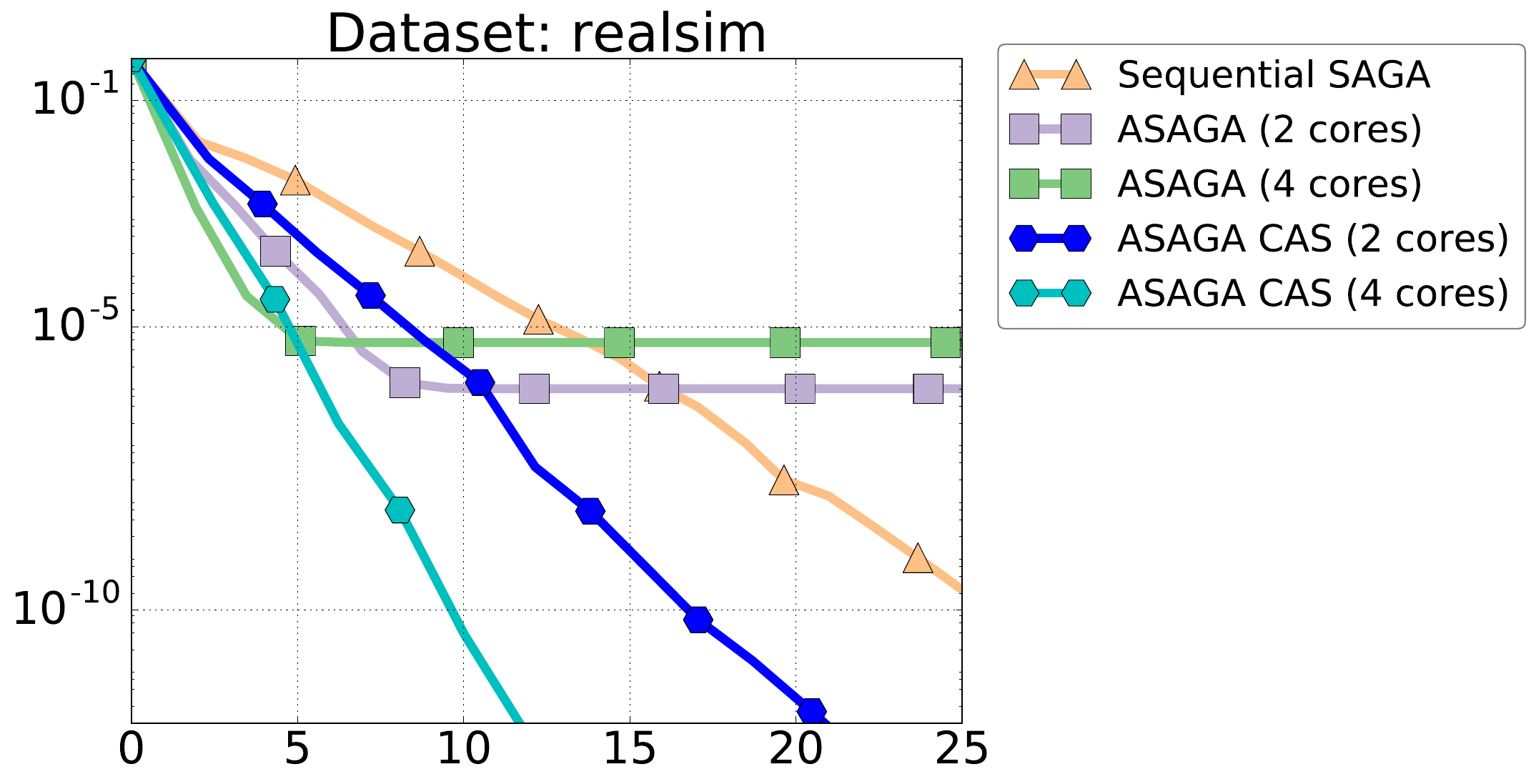}
\caption{{\bf Compare and swap in the implementation of \ASAGA}. 
Suboptimality as a function of time for \ASAGA, both using compare-and-swap (CAS) operations and using standard operations. 
The graph reveals that CAS is indeed needed in a practical implementation to ensure convergence to a high precision.} \label{fig:cas_comparison}
\end{figure}

\paragraph{Efficient storage of the $\alpha_i$.}
Storing $n$ gradient may seem like an expensive proposition, but for linear predictor models, one can actually store a single scalar per gradient (as proposed in~\citet{laggedsaga}), which is what we do in our implementation of \ASAGA.

\paragraph{Step sizes.} For each algorithm, we picked the best step size among 10 equally spaced values in a grid, and made sure that the best step size was never at the boundary of this interval. 
For Covtype and RCV1, we used the interval $[\frac{1}{10L}, \frac{10}{L}]$, whereas for URL we used  the interval $[\frac{1}{L}, \frac{100}{L}]$ as it admitted larger step sizes.
It turns out that the best step size was fairly constant for different number of cores for both \ASAGA\ and \KROMAGNON, and both algorithms had similar best step sizes.

\subsection{Biased update in the implementation}\label{apx:Bias}
In the implementation detailed in Algorithm~\ref{alg:sagasync}, $\bar \alpha$ is maintained in memory instead of being recomputed for every iteration.
This saves both the cost of reading every data point for each iteration and of computing $\bar \alpha$ for each iteration.

However, this removes the unbiasedness guarantee.
The problem here is the definition of the expectation of $\hat \alpha_i$.
Since we are sampling uniformly at random, the average of the $\hat \alpha_i$ is taken at the precise moment when we read the $\alpha_i^t$ components. 
Without synchronization, between two reads to a single coordinate in $\alpha_i$ and in $\bar \alpha$, new updates might arrive in $\bar \alpha$ that are not yet taken into account in $\alpha_i$. 
Conversely, writes to a component of $\alpha_i$ might precede the corresponding write in $\bar \alpha$ and induce another source of bias. 

In order to alleviate this issue, we can use coordinate-level locks on $\alpha_i$ and $\bar \alpha$ to make sure they are always synchronized. 
Such low-level locks are quite inexpensive when $d$ is large, especially when compared to vector-wide locks.

However, as previously noted, experimental results indicate that this fix is not necessary.

\end{document}